\documentclass[letterpaper, 10 pt, conference]{ieeeconf}  

\IEEEoverridecommandlockouts                              
\def\bse{{\boldsymbol{e}}}

\def\bsv{{\boldsymbol{v}}}

\def\bsx{{\boldsymbol{x}}}
\def\bsy{{\boldsymbol{y}}}

\usepackage{epsfig} 
\usepackage{epstopdf}

\usepackage{fancyhdr}
\usepackage{indentfirst}
\usepackage{amsmath}
\usepackage{amssymb}

\usepackage{amsthm}

\usepackage{graphicx}
\usepackage{color}

\definecolor{gray}{RGB}{128,128,128}

\usepackage{latexsym}
\usepackage{amsfonts}
\usepackage{amssymb,amsmath,amsfonts}
\usepackage{amsmath,mathrsfs,bm,url,times}
\usepackage{cite}
\usepackage{caption}
\usepackage{subcaption}
\newtheorem{assumption}{Assumption}

\newtheorem{theorem}{Theorem}
\newtheorem{lemma}{Lemma}

\newtheorem{remark}{Remark}

\DeclareMathOperator{\diag}{diag}
\DeclareMathOperator{\Deg}{Deg}
\DeclareMathOperator*{\argmin}{arg\,min}
\DeclareMathOperator{\conv}{conv}
\allowdisplaybreaks
\title{\LARGE \bf
Distributed  Optimization for Second-Order Multi-Agent Systems with Dynamic Event-Triggered Communication
}

\author{Xinlei Yi, Lisha Yao, Tao Yang, Jemin George, and Karl H. Johansson
\thanks{This work was supported by the
Knut and Alice Wallenberg Foundation, the  Swedish Foundation for Strategic Research, the Swedish Research Council, and  the Ralph E. Powe Junior Faculty Enhancement Award for the Oak Ridge Associated Universities (ORAU).}
\thanks{X. Yi and K. H. Johansson are with the Department of Automatic Control, School of Electrical Engineering and Computer Science, KTH Royal Institute of Technology, 100 44, Stockholm, Sweden.
        {\tt\small \{xinleiy, kallej\}@kth.se}.}%
\thanks{L. Yao and T. Yang are with the Department of Electrical Engineering, University of North Texas, Denton, TX 76203 USA. {\tt\small LishaYao@my.unt.edu}, {\tt\small Tao.Yang@unt.edu}.}
\thanks{J. George is with the U.S. Army Research Laboratory, Adelphi, MD 20783, USA. {\tt\small jemin.george.civ@mail.mil}.}
}

\begin{document}

\maketitle
\thispagestyle{empty}
\pagestyle{empty}

\begin{abstract}\label{cdc18so:Abstract}In this paper, we propose a fully distributed algorithm for second-order continuous-time multi-agent systems to solve the distributed optimization problem. The global objective function is a sum of private cost functions associated with the individual agents and  the interaction between agents is described by a weighted undirected graph. We show the exponential convergence of the proposed algorithm if the underlying graph is connected, each private cost function is locally  gradient-Lipschitz-continuous, and the global objective function is restricted strongly convex with respect to the global minimizer. Moreover, to reduce
the overall need of communication, we then propose a dynamic event-triggered communication mechanism that is free of Zeno behavior. It is shown that the exponential convergence is achieved if the private cost functions are also globally gradient-Lipschitz-continuous. Numerical simulations are provided to illustrate the effectiveness of the theoretical results.
\end{abstract}
\section{INTRODUCTION}\label{cdc18sosec:intro}

Distributed optimization in multi-agent systems is an important class of distributed optimization problems and has received great attention in recent years due to its wide application in wireless networks, sensor networks, smart grids, and multi-robot systems.

From a control point of view, distributed convex optimization in multi-agent systems is the optimal consensus problem, where the global objective function is a sum of private convex cost functions associated with the individual agents and  the interaction between agents is described by a graph.
Although classical distributed algorithms based on consensus theory and (sub)gradient method are discrete-time \cite{nedic2009distributed,nedic2017network,yang2018global}, continuous-time algorithms have attracted much attention recently due to the development of cyber-physical systems and the well-developed continuous-time control techniques. 
For example \cite{shi2011multi,wang2011control,lu2012zero,gharesifard2014distributed,zhang2014distributed,liu2015second,qiu2016distributed,yu2016gradient,xie2017global,zeng2017distributed}
propose continuous-time distributed algorithms to solve (constrained or unconstrained) optimal consensus problems and analyze the convergence properties via classic stability analysis.

However, all these existing continuous-time algorithms require continuous information exchange between agents, which may be impractical in physical
applications. The event-triggered  communication and control mechanism is introduced partially to tackle this problem \cite{aastrom1999comparison,heemels2012introduction}. Event-triggered  communication and control mechanisms for multi-agent systems have been studied recently \cite{dimarogonas2012distributed,seyboth2013event,meng2015periodic,yi2016distributed,yi2017pull,yi2017formation,yi2017distributed}.
Key challenges are how to design the control law, determine the event threshold, and avoid Zeno behavior. Zeno behavior means that there are an infinite number of triggers in a
finite time interval \cite{johansson1999regularization}.

There are few works on the optimal consensus problem with event-triggered  communication.  In \cite{kia2015distributed}, the authors design a distributed continuous-time algorithm for first-order multi-agent systems with event-triggered  communication. In \cite{chen2016event}, the authors extend the zero-gradient-sum algorithm proposed in \cite{lu2012zero} with event-triggered  communication. In \cite{Tran2018distributed}, the authors propose a distributed continuous-time algorithm for second-order multi-agent systems with event-triggered  communication.
However, these algorithms are not fully distributed since the gain parameters of the algorithms depends on some global parameters, such as the eigenvalues of the graph Laplacian matrix.

In this paper, we consider the distributed optimization problem for second-order multi-agent systems with undirected and connected topologies.
In particular, double-integrator dynamics are considered since they are widely applied to mechanical systems. For example, Euler-Lagrangian systems with exact knowledge of nonlinearities can be converted into double integrators and they can be used to describe many mechanical systems, such as autonomous vehicles, see \cite{qiu2016optimal,zhang2017distributed}. Moreover, the considered distributed optimization problem has many applications, such as the targeted agreement problem for a group of Lagrangian systems \cite{meng2017targeted}.
A fully distributed continuous-time algorithm is first proposed to solve the problem.
One related existing work is \cite{zhang2014distributed}, which also proposes a continuous-time distributed algorithm for second-order multi-agent systems.
However, in \cite{zhang2014distributed}, the parameters of the algorithm depend on some global information and the speed information of each agent has to be exchanged between neighbors, and only asymptotic convergence is established for the case where private cost functions are strongly convex and globally  gradient-Lipschitz-continuous.
In contrast, in this paper, no global information is needed to be known in advance and each agent does not need its neighbors' speed information.
For the case where private cost functions are convex, we show the asymptotic convergence.
Furthermore, we establish the exponential convergence for the case where each private cost function is locally  gradient-Lipschitz-continuous and the global objective function is restricted strongly convex  with respect to the global minimizer.
Note that not all private cost functions need to be so or strongly convex, which is a less restricted condition compared with that in \cite{zhang2014distributed}.
To reduce the overall need of communication, inspired by the distributed dynamic event-triggered control mechanism for multi-agent systems proposed in \cite{yi2017distributed}, we then extend our algorithm with dynamic event-triggered communication.
The proposed dynamic event-triggered communication mechanism is also fully distributed since no global information, such as the Laplacian matrix, is required.
We show that the proposed dynamic event-triggered communication mechanism is free of Zeno behavior by a contradiction argument.
Moreover, we also show that the extended algorithm with the proposed event-triggered communication  mechanism exponentially converges to the global minimizer when each private cost function is globally  gradient-Lipschitz-continuous and the global objective function is restricted strongly convex.


The rest of this paper is organized as follows. \mbox{Section \ref{cdc18so:sec:preliminaries}} introduces the preliminaries. The main results are stated in Sections \ref{cdc18so:sec:cont} and \ref{cdc18so:sec:event}. Simulations are given in Section \ref{cdc18so:sec:simulation}. Finally, the paper is concluded in Section \ref{cdc18so:sec:conclusion}.

\noindent {\bf Notations}: $\|\cdot\|$ represents the Euclidean norm for
vectors or the induced 2-norm for matrices. ${\bf 1}_n$ denotes the column
vector with each component being 1 and dimension $n$. $I_n$ is the $n$-dimensional identity matrix.
Given a vector $[x_1,\dots,x_n]^\top\in\mathbb{R}^n$, $\diag([x_1,\dots,x_n])$ is a diagonal matrix with the $i$-th diagonal element being $x_i$. The notation $A\otimes B$ denotes the Kronecker product
of matrices $A$ and $B$. $\rho(\cdot)$ stands for the spectral radius for matrices and $\rho_2(\cdot)$ indicates the minimum
positive eigenvalue for matrices having positive eigenvalues. Given two
symmetric matrices $M,N$, $M\ge N$ means that $M-N$ is positive semi-definite. 

\section{PRELIMINARIES}\label{cdc18so:sec:preliminaries}
In this section, we present some definitions from algebraic graph theory \cite{mesbahi2010graph} and  the problem formulation.

\subsection{Algebraic Graph Theory}

Let $\mathcal G=(\mathcal V,\mathcal E, A)$ denote a weighted undirected graph with the set of vertices (nodes) $\mathcal V =\{1,\dots,n\}$, the set of links (edges) $\mathcal E
\subseteq \mathcal V \times \mathcal V$, and the weighted adjacency matrix
$A =A^{\top}=(a_{ij})$ with nonnegative elements $a_{ij}$. A link of $\mathcal G$ is denoted by $(i,j)\in \mathcal E$ if $a_{ij}>0$, i.e., if vertices $i$ and $j$ can communicate with each other. It is assumed that $a_{ii}=0$ for all $i\in \mathcal V$. Let $\mathcal{N}_i=\{j\in \mathcal V\mid a_{ij}>0\}$ and $\deg_i=\sum\limits_{j=1}^{n}a_{ij}$ denotes the neighbor index set and weighted degree of vertex $i$, respectively. The degree matrix of graph $\mathcal G$ is $\Deg=\diag([\deg_1, \cdots, \deg_n])$. The Laplacian matrix is $L=(L_{ij})=\Deg-A$. A  path of length $k$ between vertices $i$ and $j$ is a subgraph with distinct vertices $i_0=i,\dots,i_k=j\in\mathcal V$ and edges $(i_j,i_{j+1})\in\mathcal E,~j=0,\dots,k-1$.
An undirected graph is  connected if there exists at least one path between any two vertices.

\subsection{Problem Formulation}
Consider a network of $n$ agents and the underlying interaction between agents is described by a weighted undirected graph $\mathcal G=(\mathcal V,\mathcal E, A)$.
Each agent is described by a double integrator
\begin{align}\label{cdc18so:doublesystem}
\ddot{x}_i(t)=u_i(t),~i\in\mathcal V,~t\ge0,
\end{align}
where $x_i\in\mathbb{R}^p$ is the state and $u_i\in\mathbb{R}^p$ is the control input of agent $i$.
Each agent $i$ is also associated with a private convex cost function $f_i(x_i):~\mathbb{R}^p\mapsto\mathbb{R}$.

The goal of the distributed optimization problem is to design an algorithm, i.e., design the control input $u_i$ for every agent, so that all agents find an optimizer $x^*$ that minimizes the sum of the $f_i$'s collaboratively in a distributed manner, i.e.,
\begin{equation}\label{cdc18so:eqn:xopt}
  x^* \in \argmin_{x \in \mathbb{R}^p} \sum^{n}_{i=1} f_i(x).
\end{equation}
The existence  of the global minimizer $x^*$ is guaranteed by the following assumption.
\begin{assumption}\label{cdc18so:ass:filc} (Convex)
For each $i\in \mathcal{V}$, the function $f_i$ is continuously differentiable and  convex.
\end{assumption}

Moreover, if the following assumption also holds, then the global minimizer $x^*$ is unique.
\begin{assumption}\label{cdc18so:ass:fil} (Restricted strongly convex, see \cite{shi2015extra}) The global objective function $\sum_{i=1}^nf_i(x)$ is restricted strongly convex with respect to its global minimizer $x^*$ with convexity parameter $m_f>0$, i.e., for all $x\in\mathbb{R}^p$,
\begin{align*}
\sum_{i=1}^n(\nabla f_i(x)-\nabla f_i(x^*))^\top(x-x^*)\ge m_f\|x-x^*\|^2.
\end{align*}
\end{assumption}

\begin{remark}
Assumption \ref{cdc18so:ass:fil} is weaker than the assumption that the global object function is strongly convex, thus it is also weaker than the assumption that each private convex cost function is strongly convex.
\end{remark}

In addition, same as the existing literature, we assume that each  private cost  function has a locally (globally) Lipschitz continuous gradient.
\begin{assumption}\label{cdc18so:ass:fiu} (Locally  gradient-Lipschitz-continuous)
For each $i\in \mathcal{V}$, for any compact set $D\subseteq\mathbb{R}^p$, there exists a constant $M_i(D)>0$, such that $\|\nabla f_i(a)-\nabla f_i(b)\|\le M_i(D)\|a-b\|,~\forall a,b\in D$.
\end{assumption}
\begin{assumption}\label{cdc18so:ass:fiug} (Globally  gradient-Lipschitz-continuous)
For each $i\in \mathcal{V}$,  there exists a constant $\overline{M}_i>0$, such that $\|\nabla f_i(a)-\nabla f_i(b)\|\le \overline{M}_i\|a-b\|,~\forall a,b\in\mathbb{R}^p$.
\end{assumption}

\section{Distributed Continuous-Time Algorithms}\label{cdc18so:sec:cont}
In this section, we propose a distributed continuous-time algorithm to solve the optimization problem stated in \eqref{cdc18so:eqn:xopt} and analyze its convergence.

For each agent $i\in\mathcal V$, we first design the following algorithm,
\begin{subequations}\label{cdc18so:eqn:ctau}
\begin{align}
  \dot{v}_i(t) = &\beta \sum_{j=1}^{n} L_{ij} x_j(t),~\sum_{i=1}^nv_i(0)=0,\label{cdc18so:eqn:ctau1}\\
  u_i(t) = & - \gamma \dot{x}_i(t)-\alpha \beta \sum_{j=1}^{n} L_{ij}x_j(t) - \theta v_i(t)\nonumber\\
  & - \alpha \nabla f_i(x_i(t)), ~t\ge0, \label{cdc18so:eqn:ctau2}
\end{align}
\end{subequations}
where $\alpha>0$, $\beta>0$, $\gamma>0$, and $\theta>0$ are gain parameters.
\begin{remark}
In the design of right-hand side in \eqref{cdc18so:eqn:ctau2}, $- \gamma \dot{x}_i(t)$ is to ensure the convergence of \eqref{cdc18so:eqn:ctau}, $-\alpha\beta \sum_{j=1}^{n} L_{ij}x_j(t)$ is to ensure the consensus among agents, $- \alpha \nabla f_i(x_i(t))$ is to optimize each agent's private cost function, and $-\theta v_i(t)$ together with $\sum_{i=1}^nv_i(0)=0$ and \eqref{cdc18so:eqn:ctau1} are to maintain the equilibrium point at the optimal point. Moreover, by setting $v_i(0)=0,~\forall i\in\mathcal{V}$, the coordination between agents to let $\sum_{i=1}^nv_i(0)=0$ can be avoided.
\end{remark}

Denote $y_i(t)=\dot{x}_i(t)$. Then we can rewrite \eqref{cdc18so:doublesystem} and \eqref{cdc18so:eqn:ctau} as
\begin{subequations}\label{cdc18so:eqn:ctauy}
\begin{align}
    \dot{x}_i(t) =&y_i(t),~\forall x_i(0),~t\ge0,\label{cdc18so:eqn:ctauy2}\\
  \dot{y}_i(t) =  & - \gamma y_i(t)-\alpha\beta \sum_{j=1}^{n} L_{ij} x_j(t) -\theta  v_i(t)\nonumber\\
  & - \alpha \nabla f_i(x_i(t)),~\forall y_i(0), \label{cdc18so:eqn:ctauy3}\\
  \dot{v}_i(t) = &\beta \sum_{j=1}^{n} L_{ij} x_j(t),~\sum_{i=1}^nv_i(0)=0.\label{cdc18so:eqn:ctauy1}
\end{align}
\end{subequations}

\begin{remark}
If there is only one agent, the algorithm \eqref{cdc18so:eqn:ctauy} becomes the heavy ball with friction system \cite{attouch2000heavy}:
\begin{align*}
\ddot{x}+\gamma \dot{x}+\alpha\nabla f(x)=0.
\end{align*}
\end{remark}

Denote $\bsx=[x_1^\top,\cdots,x_n^\top]^\top$, $\bsy=[y_1^\top,\cdots,y_n^\top]^\top$, $\bsv=[v_1^\top,\cdots,v_n^\top]^\top$, and $f(\bsx)=\sum_{i=1}^nf_i(x_i)$. Then, we can rewrite \eqref{cdc18so:eqn:ctauy} in the following compact form:
\begin{subequations}\label{cdc18so:eqn:cta}
\begin{align}
    \dot{\bsx}(t) =&\bsy(t),~\forall x(0),~t\ge0,\label{cdc18so:eqn:cta2}\\
  \dot{\bsy}(t) = & - \gamma \bsy(t)-\alpha\beta (L \otimes I_p)\bsx(t) -\theta  \bsv(t)\nonumber\\
  & - \alpha \nabla f(\bsx(t)),~\forall \bsy(0),\label{cdc18so:eqn:cta3}\\
   \dot{\bsv}(t) = &\beta ( L \otimes I_p ) \bsx(t),~\sum_{i=1}^nv_i(0)=0,\label{cdc18so:eqn:cta1}
\end{align}
\end{subequations}

The following result establishes sufficient conditions on the private cost function $f_i$; the gain parameters $\alpha,~\gamma,~\theta$; and the underlying graph to guarantee the (exponential) convergence of \eqref{cdc18so:eqn:ctauy}. 
\begin{theorem}\label{cdc18so:theorem}
Suppose that Assumption \ref{cdc18so:ass:filc}  holds, and that the underlying undirected graph $\mathcal{G}$ is connected. If every agent $i\in\mathcal{V}$ runs the distributed algorithm with continuous-time communication given in \eqref{cdc18so:eqn:ctauy} and $\theta<\alpha\gamma$, then every individual solution $x_i(t)$ asymptotically converges to one global minimizer. Moreover, if Assumptions \ref{cdc18so:ass:fil} and  \ref{cdc18so:ass:fiu} are also satisfied, then every individual solution $x_i(t)$ exponentially converges to the unique global minimizer $x^*$ with a rate no less than $\frac{\varepsilon_3}{2\varepsilon_4}$, where
\begin{align}
\varepsilon_1=&\min\{\gamma(1-\varepsilon_0),~\alpha\gamma\varepsilon_0m_1\}>0,\label{cdc18so:varepsilon1}\\
\varepsilon_2=&\max\Big\{\frac{\gamma}{\alpha}+\frac{\gamma^2}{\theta}
+\frac{\theta}{\alpha^2},~\frac{\alpha^2(M(D))^2}{\theta }\Big\}>0,\label{cdc18so:varepsilon2}\\
\varepsilon_3=&\min\Big\{\varepsilon_1,~\frac{\varepsilon\theta }{2}\Big\}>0,
\label{cdc18so:varepsilon3}\\
\varepsilon_4=&
\max\Big\{1+\frac{\varepsilon\varepsilon_2}{\varepsilon_1}+\frac{\varepsilon}{\alpha},\nonumber\\
&(1+\frac{\varepsilon\varepsilon_2}{\varepsilon_1})(\gamma^2\varepsilon_0+
\alpha\beta\rho(L)+\frac{\alpha M(D)}{2})+\frac{\varepsilon M(D)}{2},\nonumber\\
&(1+\frac{\varepsilon\varepsilon_2}{\varepsilon_1})
\frac{\theta\gamma\varepsilon_0}{\beta\rho_2(L)}+\varepsilon\alpha\Big\}>1,
\label{cdc18so:varepsilon4}
\end{align}
where $\varepsilon>0$ and $\varepsilon_0\in(\frac{\theta}{\alpha\gamma},1)$ are design parameters  and can be freely chosen in the given intervals, $m_1=\min\Big\{\frac{m_f}{2},\frac{\rho_2(L)m_f^2\alpha\gamma\varepsilon_0}
{2(\alpha\gamma\varepsilon_0-\theta)(m^2_f+16M^2(D))}\Big\}>0$ and $M(D)=\max_{i\in \mathcal{V}}\{M_i(D)\}>0$ are constants, and $D\subseteq\mathbb{R}^p$ is a compact convex set and its definition is given in the proof.
\end{theorem}

\begin{proof}
Due to the  space limitations, the proof is omitted here, but can be found in \cite{yi2018distributedopti}.
The proof is based the Lyapunov stability analysis. A novel Lyapunov function is constructed, which is different from those in the existing literature.
\end{proof}

\begin{remark}
The algorithm \eqref{cdc18so:eqn:ctauy} is fully distributed in the sense that it does not require any global parameters to design the gain parameters $\alpha$, $\beta$, $\gamma$, and $\theta$. On the other hand, the algorithms proposed in \cite{zhang2014distributed,yu2016gradient} do not have such a property.
\end{remark}

\begin{remark} We could also construct an alternative algorithm:
\begin{subequations}\label{cdc18so:eqn:ctaLv}
\begin{align}
    \dot{x}_i(t) =&y_i(t),~\forall x_i(0),~t\ge0,\label{cdc18so:eqn:ctaLv2}\\
  \dot{y}_i(t) = & - \gamma y_i(t)-\alpha\beta \sum_{j=1}^nL_{ij}x_j(t)\nonumber\\
  &  - \theta \sum_{j=1}^nL_{ij}v_j(t)- \alpha \nabla f_i(x_i(t)),~\forall y_i(0),\label{cdc18so:eqn:ctaLv3}\\
    \dot{v}_i(t) = &\beta\sum_{j=1}^nL_{ij}x_i(t),~\forall v_i(0).\label{cdc18so:eqn:ctaLv1}
\end{align}
\end{subequations}
Similar results as shown in Theorem \ref{cdc18so:theorem} could be given and proven. We omit the details due to space limitations. 

Different from the requirement that $\sum_{i=1}^nv_i(0)=0$ in the algorithm \eqref{cdc18so:eqn:cta}, $v_i(0)$ can be arbitrarily chosen in the algorithm  \eqref{cdc18so:eqn:ctaLv}. In other words, the algorithm  \eqref{cdc18so:eqn:ctaLv} is robust to the initial condition $v_i(0)$. However, the algorithm  \eqref{cdc18so:eqn:ctaLv} requires additional communication of $v_j$ in \eqref{cdc18so:eqn:ctaLv3}, compared to the algorithm \eqref{cdc18so:eqn:cta}.
\end{remark}




\section{EVENT-TRIGGERED COMMUNICATION}\label{cdc18so:sec:event}
To implement the distributed algorithm \eqref{cdc18so:eqn:ctauy}, every agent $i\in\mathcal{V}$ has to know the continuous-time state $x_j(t),~\forall j\in\mathcal{N}_i$. In other words, continuous communication between agents is needed. However, distributed networks are normally resources-constrained and communication is energy-consuming. To avoid continuous communication, inspired by the idea of event-triggered control for multi-agent systems \cite{dimarogonas2012distributed}, we consider event-triggered communication.
More specifically, we extend the algorithm \eqref{cdc18so:eqn:ctauy} with event-triggered communication mechanism as: 
\begin{subequations}\label{cdc18so:eqn:ctauyevent}
\begin{align}
    \dot{x}_i(t) =&y_i(t),~\forall x_i(0),~t\ge0,\label{cdc18so:eqn:ctauyevent2}\\
  \dot{y}_i(t) =  & - \gamma y_i(t)-\alpha\beta \sum_{j=1}^{n} L_{ij} x_j(t^j_{k_j(t)})  -\theta  v_i(t)\nonumber\\
  &- \alpha \nabla f_i(x_i(t)),~\forall y_i(0),\label{cdc18so:eqn:ctauyevent3}\\
  \dot{v}_i(t) = &\beta \sum_{j=1}^{n} L_{ij} x_j(t^j_{k_j(t)}),\nonumber\\
&t\in[t^i_k,t^i_{k+1}),~k=1,2,\dots,~\sum_{i=1}^nv_i(0)=0,\label{cdc18so:eqn:ctauyevent1}
\end{align}
\end{subequations}
where the increasing sequence  $\{t_{k}^{i}\}_{k=1}^{\infty}$, $\forall i\in\mathcal V$ to be determined later is the triggering times and $t^j_{k_j(t)}=\max\{t^{j}_{k}:~t^{j}_{k}\le t\}$. We assume $t^j_1=0, ~\forall j\in\mathcal V$. For simplicity, let $\hat{x}_j(t)=x_j(t^j_{k_j(t)})$ and $e^{x}_j(t)=\hat{x}_j(t)-x_j(t)$.

Denote $\hat{\bsx}=[\hat{x}_1^\top,\cdots,\hat{x}_n^\top]^\top$ and $\bse^{\bsx}=[(e^{x}_1)^\top,\cdots,(e^{x}_n)^\top]^\top$. Then, we can rewrite \eqref{cdc18so:eqn:ctauyevent} in the following compact form:
\begin{subequations}\label{cdc18so:eqn:ctaevent}
\begin{align}
    \dot{\bsx}(t) =&\bsy(t),~\forall \bsx(0),~t\ge0,\label{cdc18so:eqn:ctaevent2}\\
  \dot{\bsy}(t) = & - \gamma \bsy(t)-\alpha\beta (L \otimes I_p)\hat{\bsx}(t) -\theta  \bsv(t)\nonumber\\
  & - \alpha \nabla f(\bsx(t)),~\forall \bsy(0),\label{cdc18so:eqn:ctaevent3}\\
   \dot{\bsv}(t) = &\beta ( L \otimes I_p ) \hat{\bsx}(t),~\sum_{i=1}^nv_i(0)=0.\label{cdc18so:eqn:ctaevent1}
\end{align}
\end{subequations}

In the following theorem, we propose a dynamic event-triggered law to determine the triggering times such that the solution of the distributed optimization problem can still be reached exponentially.
\begin{theorem}\label{cdc18so:theoremvevent}
Suppose that Assumptions \ref{cdc18so:ass:filc}, \ref{cdc18so:ass:fil}, and \ref{cdc18so:ass:fiug} hold,  and that the underlying undirected graph $\mathcal{G}$ is connected. Suppose that each agent $i\in\mathcal{V}$ runs the distributed algorithm with event-triggered communication given in \eqref{cdc18so:eqn:ctauyevent} and $\theta<\alpha\gamma$.
Given the first triggering time $t^i_1=0$, every agent $i \in \mathcal{V}$ determines the triggering times $\{t^i_k\}_{k=2}^{\infty}$ by the following dynamic event-triggered law:
\begin{align}
t^i_{k+1}=&\min \Big\{t:~\kappa_i\Big(\|e^{x}_i(t)\|^2 -\frac{(\alpha\gamma\varepsilon_0-\theta)\beta\sigma_i}{4\varphi_i}\hat{q}_i(t)\Big)\nonumber\\
 &~~~~~~~~~~~\ge  \chi_i(t),~t \ge t^i_k \Big\},~k=1,2,\ldots \label{cdc18so:dynamicv}\\
\hat{q}_i(t)=&-\frac{1}{2} \sum_{j \in \mathcal{N}_i} L_{ij}\|\hat{x}_j(t)-\hat{x}_i(t)\|^2 \geq 0,\label{cdc18so:qhatxy}\\
\dot{\chi}_i(t)=&-\delta_i \Big(\|e^{x}_i(t)\|^2 -\frac{(\alpha\gamma\varepsilon_0-\theta)\beta\sigma_i}{4\varphi_i}\hat{q}_i(t)\Big)\nonumber\\
&-\phi_i \chi_i(t),~\forall \chi_i(0)>0,
\label{cdc18so:chi}
\end{align}
where $\sigma_i \in [0,1)$, $\phi_i>0$, $\delta_i \in [0,1]$, and $\kappa_i>\frac{1-\delta_i}{\phi_i}$  are design parameters and can be freely chosen in the given interval; and
\begin{align}
m_2=&\min\Big\{\frac{m_f}{2},~\frac{4\rho_2(L)m_f^2\alpha}
{(\alpha\gamma\varepsilon_0-\theta)\beta(m^2_f+16\overline{M}^2)}\Big\},\label{cdc18so:m2}\\
\varepsilon_5=&\min\Big\{\frac{\gamma(1-\varepsilon_0)}{2},~m_2\alpha\Big\}>0,\label{cdc18so:varepsilon5}\\
\varepsilon_6=&\max\Big\{\frac{\gamma}{\alpha}+\frac{\gamma^2}{\theta}
+\frac{\theta}{\alpha^2},~\frac{\alpha^2\overline{M}^2}{\theta }\Big\}>0,\label{cdc18so:varepsilon6}\\
\varepsilon_7=&1+\frac{\varepsilon\varepsilon_6}{\varepsilon_5}>1,
\label{cdc18so:varepsilon7}\\
\varepsilon_8=&\frac{\varepsilon}{4\varepsilon_7}>0,\label{cdc18so:varepsilon8}\\
\varphi_i=&\frac{(\alpha\gamma\varepsilon_0-\theta)\beta}{4}L_{ii}+(\alpha\gamma\varepsilon_0-\theta)\beta L_{ii}+\frac{\gamma^2\theta
\varepsilon_0^2 }{4\varepsilon_8}\nonumber\\
&+\frac{\alpha^2\beta^2}{\gamma(1-\varepsilon_0)}\Big(L_{ii}-\sum_{j=1,j\neq i}^nL_{jj}L_{ij}\Big)
\end{align}
with $\overline{M}=\max_{i\in\mathcal{V}}\{\overline{M}_i\}$ is a constant, $\varepsilon>0$ and $\varepsilon_0\in(\frac{\theta}{\alpha\gamma},1)$ are design parameters,
then (i) there is no Zeno behavior, and (ii) every individual solution $x_i(t)$ exponentially converges to the unique global minimizer $x^*$ with a rate no less than $\frac{\varepsilon_9}{2\varepsilon_{10}}$, where
\begin{align}
\varepsilon_9=&\min\Big\{\varepsilon_5,
~\frac{\varepsilon\theta }{4},~k_d\Big\}>0,\label{cdc18so:varepsilon9}\\
\varepsilon_{10}=&\max\Big\{\varepsilon_7
+\frac{\varepsilon}{\alpha},~\varepsilon_7\Big(\gamma^2\varepsilon_0+
\alpha\beta\rho(L)+\frac{\alpha \overline{M}}{2}\Big)+\frac{\varepsilon \overline{M}}{2},\nonumber\\
&~~~~~~~~\varepsilon_7
\frac{\theta\gamma\varepsilon_0}{\beta\rho_2(L)}+\frac{\varepsilon\alpha}{\rho_2(L)}\Big\}>1,
\label{cdc18so:varepsilon10}
\end{align}
with  $k_d=\min_{i\in\mathcal{V}}\Big \{ \phi_i-\frac{1-\delta_i}{\kappa_i}\Big\}>0$.
\end{theorem}

\begin{proof}
Due to the space limitations, the proof is omitted here, but can be found in \cite{yi2018distributedopti}.
\end{proof}

\begin{remark}
The proposed dynamic event-triggered communication has several nice features: i) the exchange of $x_i(t)$ only occurs at the discrete time points $\{t^i_k,~i\in\mathcal V\}_{k=1}^{\infty}$, ii) it is free of Zeno behavior, and iii) the implementation does not require any global information such as the Laplacian matrix.
One potential drawback of the proposed dynamic event-triggered law is that when determining  $\varphi_i$ the global parameters $\rho_2(L)$, $m_f$, and $\overline{M}$ are needed.
One solution to overcome this drawback is let $\sigma_i=\delta_i=0,~i\in\mathcal V$, since in this case we do not need to know $\varphi_i$. 
\end{remark}

\begin{remark}
If we let $\delta_1=\cdots=\delta_n=0$ and $\phi_1=\cdots=\phi_n\in(0,\frac{\varepsilon_9}{\varepsilon_{10}}]$ in \eqref{cdc18so:chi}, where $\varepsilon_9$ and $\varepsilon_{10}$ is defined in \eqref{cdc18so:varepsilon9} and \eqref{cdc18so:varepsilon10}, respectively, then similar to the proof of Theorem 3.2 in \cite{seyboth2013event}, for each agent $i\in\mathcal V$, we can find a positive constant $\tau_i$, such that $t^i_{k+1}-t^i_k\ge \tau_i,~k=1,2,\dots$. Since the proof is similar, we omit the detailed analysis here.
\end{remark}

\section{SIMULATIONS}\label{cdc18so:sec:simulation}
In this section, we illustrate and validate the proposed algorithm through numerical examples and compare the results with other existing algorithms.
Consider a simple network of $n=3$ agents with the Laplacian matrix
\begin{eqnarray*}
L=\left[\begin{array}{rrrr}1&-1&0\\
-1&2&-1\\
0&-1&1
\end{array}\right].
\end{eqnarray*}

We first consider the case that the private cost functions $f_i$  and the global objective function $\sum_{i=1}^nf_i(x)$ are just convex. We choose $f_i(x)=\frac{1}{2}(x-a_i)^\top A_i(x-a_i)$,
\begin{align*}
&A_1=\left[\begin{array}{rrrr}2&-1&-1\\
-1&1.5&-0.5\\
-1&-0.5&1.5
\end{array}\right],
A_2=\left[\begin{array}{rrrr}3&-3&0\\
-3&4&-1\\
0&-1&1
\end{array}\right],\\
&A_3=\left[\begin{array}{rrrr}2.5&0&-2.5\\
0&10&-10\\
-2.5&-10&12.5
\end{array}\right],
a_1=\left[\begin{array}{r}0.6132\\
   -0.5278\\
    1.2416
\end{array}\right],\\
&a_2=\left[\begin{array}{r} -0.1576\\
   -1.3736\\
    0.8708
\end{array}\right],
a_3=\left[\begin{array}{r}-1.5685\\
   -1.8443\\
    0.2884
\end{array}\right].
\end{align*}
Fig. \ref{cdc18so:fignosc} shows the comparison between the distributed algorithm \eqref{cdc18so:eqn:ctauy} with $\alpha=\beta=2,~\gamma=6,~\theta=5$; algorithm (3) in \cite{kia2015distributed} with $\alpha=\beta=2$; algorithm (6) in \cite{zhang2014distributed} with $\alpha=\beta=2,~k=6$; and algorithm (3) in \cite{yu2016gradient} with $k=6$. It can be seen that distributed gradient descent algorithm (algorithm (3) in \cite{kia2015distributed}) cannot achieve a $O(\frac{1}{t^2})$ convergence when the global objective and all the private cost functions are just convex.

\begin{figure}
  \centering
  \includegraphics[width=\linewidth]{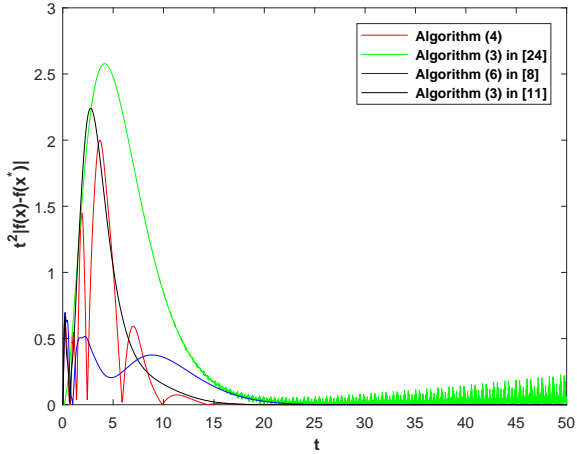}
  \caption{Simulation results for non-strongly convex private cost and global objective functions.}
  \label{cdc18so:fignosc}
\end{figure}

We then consider the case that the private cost functions $f_i$ are just convex but $\sum_{i=1}^nf_i(x)$ is strongly convex. We choose $f_i(x)=\|x-b_i\|^4$ with $x \in \mathbb{R}^3$,
\begin{eqnarray*}
b_1=\left[\begin{array}{r}0\\
0\\
0
\end{array}\right],~
b_2=\left[\begin{array}{r}2.5\\
2\\
3
\end{array}\right],~
b_3=\left[\begin{array}{r}-3.5\\
-2.7\\
-1
\end{array}\right].
\end{eqnarray*}
Fig. \ref{cdc18so:figso} shows the comparison between the distributed algorithm \eqref{cdc18so:eqn:ctauy} with $\alpha=\beta=2,~\gamma=6,~\theta=5$; algorithm (3) in \cite{kia2015distributed} with $\alpha=\beta=2$; algorithm (6) in \cite{zhang2014distributed} with $\alpha=\beta=2,~k=6$; and algorithm (3) in \cite{yu2016gradient} with $k=6$. We can see that the proposed algorithm \eqref{cdc18so:eqn:ctauy} achieves a faster convergence in this simulation.

\begin{figure}
  \centering
  \includegraphics[width=\linewidth]{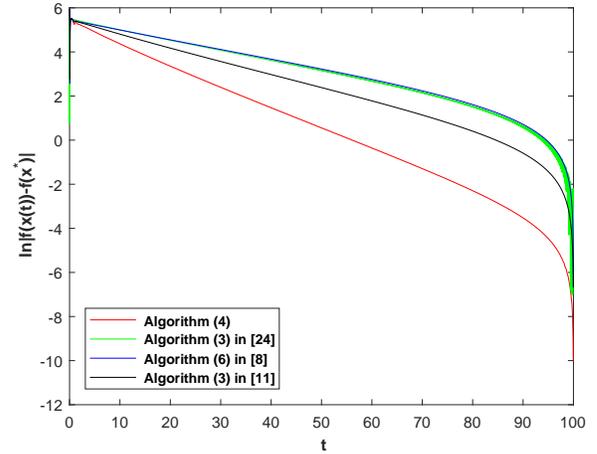}
  \caption{Simulation results for non-strongly convex private cost functions but strongly convex global objective function.}
  \label{cdc18so:figso}
\end{figure}%

Next, we consider the case where all private cost functions $f_i$ are strongly convex.
In particular, $f_i(x)=\frac{1}{2}x^\top C_ix+a_i^\top x$ with
\begin{align*}
C_1=&\left[\begin{array}{rrrr}4.7471 & 1.2843  &  0.5836 \\
    1.2843  &  5.0861 &  -2.4209 \\
    0.5836 &  -2.4209  &  2.2270
\end{array}\right],\\
C_2=&\left[\begin{array}{rrrr}1.3528  &  0.5141  & -2.1684   \\
0.5141  &  1.2333 &  -0.5857 \\
-2.1684  & -0.5857 &   4.0361
\end{array}\right],\\
C_3=&\left[\begin{array}{rrrr}1.0223  &  1.2630  & -0.4907\\
 1.2630 &   2.1391  & -0.1378\\
 -0.4907 &  -0.1378 &   0.7207
\end{array}\right].
\end{align*}
Fig. \ref{cdc18so:figsc} shows the comparison between the distributed algorithm \eqref{cdc18so:eqn:ctauy} with $\alpha=\beta=2,~\gamma=6,~\theta=3.5$; algorithm (3) in \cite{kia2015distributed} with $\alpha=\beta=2$; algorithm (6) in \cite{zhang2014distributed} with $\alpha=\beta=2,~k=6$; algorithm (3) in \cite{yu2016gradient} with $k=6$; and the distributed event-triggered algorithm \eqref{cdc18so:eqn:ctauyevent} with dynamic event-triggered communication determined by \eqref{cdc18so:dynamicv}. 
In our simulation, the sample length is $0.01$.
During time interval $[0,50]$, agents 1--3 triggered $1199$, $139$, and $664$ times, respectively, under our dynamic event-triggered communication mechanism.
Therefore, our dynamic event-triggered communication mechanism is very efficient and avoids about 85\% sampling in this simulation.
\begin{figure}
  \centering
  \includegraphics[width=\linewidth]{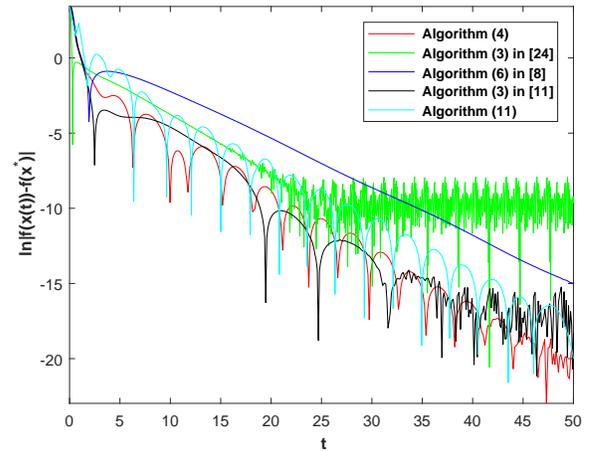}
\caption{Simulation results for strongly convex private cost functions.}
\label{cdc18so:figsc}
\end{figure}

\section{CONCLUSION}\label{cdc18so:sec:conclusion}
In this paper, we considered the distributed optimization problem for second-order continuous-time multi-agent systems.
We first proposed a fully distributed continuous-time algorithm that does not require to know any global information in advance. 
We established the asymptotic convergence when the private cost functions are convex and exponential convergence when each private cost function is locally  gradient-Lipschitz-continuous and the global objective function is restricted strongly convex  with respect to the global minimizer.
To avoid continuous communication, we then extended the continuous-time algorithm with dynamic event-triggered communication.
We again showed that the global minimizer can be reached exponentially when each private cost function is globally gradient-Lipschitz-continuous and the global objective function is restricted strongly convex.  Furthermore, the dynamic event-triggered communication was shown to be free of Zeno behavior.
Future research directions include quantifying the convergence speed when the private cost functions are just convex.

\section*{ACKNOWLEDGMENTS}
The first author is thankful to Han Zhang for discussions on distributed optimization.
\bibliographystyle{IEEEtran}
\bibliography{refs}

\begin{thebibliography}{10}
\providecommand{\url}[1]{#1}
\csname url@samestyle\endcsname
\providecommand{\newblock}{\relax}
\providecommand{\bibinfo}[2]{#2}
\providecommand{\BIBentrySTDinterwordspacing}{\spaceskip=0pt\relax}
\providecommand{\BIBentryALTinterwordstretchfactor}{4}
\providecommand{\BIBentryALTinterwordspacing}{\spaceskip=\fontdimen2\font plus
\BIBentryALTinterwordstretchfactor\fontdimen3\font minus
  \fontdimen4\font\relax}
\providecommand{\BIBforeignlanguage}[2]{{%
\expandafter\ifx\csname l@#1\endcsname\relax
\typeout{** WARNING: IEEEtran.bst: No hyphenation pattern has been}%
\typeout{** loaded for the language `#1'. Using the pattern for}%
\typeout{** the default language instead.}%
\else
\language=\csname l@#1\endcsname
\fi
#2}}
\providecommand{\BIBdecl}{\relax}
\BIBdecl

\bibitem{nedic2009distributed}
A.~Nedic and A.~Ozdaglar, ``Distributed subgradient methods for multi-agent
  optimization,'' \emph{IEEE Transactions on Automatic Control}, vol.~54,
  no.~1, pp. 48--61, 2009.

\bibitem{nedic2017network}
A.~Nedi{\'c}, A.~Olshevsky, and M.~G. Rabbat, ``Network topology and
  communication-computation tradeoffs in decentralized optimization,''
  \emph{Proceedings of the IEEE}, vol. 106, no.~5, pp. 953--976, 2018.

\bibitem{yang2018global}
T.~Yang, Y.~Wan, H.~Wang, and Z.~Lin, ``Global optimal consensus for
  discrete-time multi-agent systems with bounded controls,'' \emph{Automatica},
  vol.~97, pp. 182--185, 2018.

\bibitem{shi2011multi}
G.~Shi, K.~H. Johansson, and Y.~Hong, ``Multi-agent systems reaching optimal
  consensus with directed communication graphs,'' in \emph{American Control
  Conference}, 2011, pp. 5456--5461.

\bibitem{wang2011control}
J.~Wang and N.~Elia, ``A control perspective for centralized and distributed
  convex optimization,'' in \emph{IEEE Conference on Decision and Control and
  European Control Conference}, 2011, pp. 3800--3805.

\bibitem{lu2012zero}
J.~Lu and C.~Y. Tang, ``Zero-gradient-sum algorithms for distributed convex
  optimization: The continuous-time case,'' \emph{IEEE Transactions on
  Automatic Control}, vol.~57, no.~9, pp. 2348--2354, 2012.

\bibitem{gharesifard2014distributed}
B.~Gharesifard and J.~Cort{\'e}s, ``Distributed continuous-time convex
  optimization on weight-balanced digraphs,'' \emph{IEEE Transactions on
  Automatic Control}, vol.~59, no.~3, pp. 781--786, 2014.

\bibitem{zhang2014distributed}
Y.~Zhang and Y.~Hong, ``Distributed optimization design for second-order
  multi-agent systems,'' in \emph{Chinese Control Conference}, 2014, pp.
  1755--1760.

\bibitem{liu2015second}
Q.~Liu and J.~Wang, ``A second-order multi-agent network for bound-constrained
  distributed optimization,'' \emph{IEEE Transactions on Automatic Control},
  vol.~60, no.~12, pp. 3310--3315, 2015.

\bibitem{qiu2016distributed}
Z.~Qiu, S.~Liu, and L.~Xie, ``Distributed constrained optimal consensus of
  multi-agent systems,'' \emph{Automatica}, vol.~68, pp. 209--215, 2016.

\bibitem{yu2016gradient}
W.~Yu, P.~Yi, and Y.~Hong, ``A gradient-based dissipative continuous-time
  algorithm for distributed optimization,'' in \emph{Chinese Control
  Conference}, 2016, pp. 7908--7912.

\bibitem{xie2017global}
Y.~Xie and Z.~Lin, ``Global optimal consensus for multi-agent systems with
  bounded controls,'' \emph{Systems \& Control Letters}, vol. 102, pp.
  104--111, 2017.

\bibitem{zeng2017distributed}
X.~Zeng, P.~Yi, and Y.~Hong, ``Distributed continuous-time algorithm for
  constrained convex optimizations via nonsmooth analysis approach,''
  \emph{IEEE Transactions on Automatic Control}, vol.~62, no.~10, pp.
  5227--5233, 2017.

\bibitem{aastrom1999comparison}
K.~J. {\AA}str{\"o}m and B.~Bernhardsson, ``Comparison of periodic and event
  based sampling for first-order stochastic systems,'' in \emph{Proceedings of
  the 14th IFAC World Congress}, vol.~11, 1999, pp. 301--306.

\bibitem{heemels2012introduction}
W.~Heemels, K.~H. Johansson, and P.~Tabuada, ``An introduction to
  event-triggered and self-triggered control,'' in \emph{IEEE Conference on
  Decision and Control}, 2012, pp. 3270--3285.

\bibitem{dimarogonas2012distributed}
D.~V. Dimarogonas, E.~Frazzoli, and K.~H. Johansson, ``Distributed
  event-triggered control for multi-agent systems,'' \emph{IEEE Transactions on
  Automatic Control}, vol.~57, no.~5, pp. 1291--1297, 2012.

\bibitem{seyboth2013event}
G.~S. Seyboth, D.~V. Dimarogonas, and K.~H. Johansson, ``Event-based
  broadcasting for multi-agent average consensus,'' \emph{Automatica}, vol.~49,
  no.~1, pp. 245--252, 2013.

\bibitem{meng2015periodic}
X.~Meng, L.~Xie, Y.~C. Soh, C.~Nowzari, and G.~J. Pappas, ``Periodic
  event-triggered average consensus over directed graphs,'' in \emph{IEEE
  Conference on Decision and Control}, 2015, pp. 4151--4156.

\bibitem{yi2016distributed}
X.~Yi, W.~Lu, and T.~Chen, ``Distributed event-triggered consensus for
  multi-agent systems with directed topologies,'' in \emph{Chinese Control and
  Decision Conference}, 2016, pp. 807--813.

\bibitem{yi2017pull}
------, ``Pull-based distributed event-triggered consensus for multiagent
  systems with directed topologies,'' \emph{IEEE Transactions on Neural
  Networks and Learning Systems}, vol.~28, no.~1, pp. 71--79, Jan 2017.

\bibitem{yi2017formation}
X.~Yi, J.~Wei, D.~V. Dimarogonas, and K.~H. Johansson, ``Formation control for
  multi-agent systems with connectivity preservation and event-triggered
  controllers,'' \emph{IFAC-PapersOnLine}, vol.~50, no.~1, pp. 9367--9373,
  2017.

\bibitem{yi2017distributed}
X.~Yi, K.~Liu, D.~V. Dimarogonas, and K.~H. Johansson, ``Distributed dynamic
  event-triggered control for multi-agent systems,'' in \emph{IEEE Conference
  on Decision and Control}, 2017, pp. 6683--6698.

\bibitem{johansson1999regularization}
K.~H. Johansson, M.~Egerstedt, J.~Lygeros, and S.~Sastry, ``On the
  regularization of {Z}eno hybrid automata,'' \emph{Systems \& Control
  Letters}, vol.~38, no.~3, pp. 141--150, 1999.

\bibitem{kia2015distributed}
S.~S. Kia, J.~Cort{\'e}s, and S.~Mart{\'\i}nez, ``Distributed convex
  optimization via continuous-time coordination algorithms with discrete-time
  communication,'' \emph{Automatica}, vol.~55, pp. 254--264, 2015.

\bibitem{chen2016event}
W.~Chen and W.~Ren, ``Event-triggered zero-gradient-sum distributed consensus
  optimization over directed networks,'' \emph{Automatica}, vol.~65, pp.
  90--97, 2016.

\bibitem{Tran2018distributed}
N.-T.~T. et~al., ``Distributed optimization problem for secondorder multi-agent
  systems with event-triggered and time-triggered communication,''
  \emph{Journal of the Franklin Institute}, 2018.

\bibitem{qiu2016optimal}
Z.~Qiu, Y.~Hong, and L.~Xie, ``Optimal consensus of euler-lagrangian systems
  with kinematic constraints,'' \emph{IFAC-PapersOnLine}, vol.~49, no.~22, pp.
  327--332, 2016.

\bibitem{zhang2017distributed}
Y.~Zhang, Z.~Deng, and Y.~Hong, ``Distributed optimal coordination for multiple
  heterogeneous {E}uler--{L}agrangian systems,'' \emph{Automatica}, vol.~79,
  pp. 207--213, 2017.

\bibitem{meng2017targeted}
Z.~Meng, T.~Yang, G.~Shi, D.~V. Dimarogonas, Y.~Hong, and K.~H. Johansson,
  ``Targeted agreement of multiple {L}agrangian systems,'' \emph{Automatica},
  vol.~84, pp. 109--116, 2017.

\bibitem{mesbahi2010graph}
M.~Mesbahi and M.~Egerstedt, \emph{Graph Theoretic Methods in Multiagent
  Networks}.\hskip 1em plus 0.5em minus 0.4em\relax Princeton University Press,
  2010.

\bibitem{shi2015extra}
W.~Shi, Q.~Ling, G.~Wu, and W.~Yin, ``Extra: An exact first-order algorithm for
  decentralized consensus optimization,'' \emph{SIAM Journal on Optimization},
  vol.~25, no.~2, pp. 944--966, 2015.

\bibitem{attouch2000heavy}
H.~Attouch, X.~Goudou, and P.~Redont, ``The heavy ball with friction method,
  {I}. {T}he continuous dynamical system: global exploration of the local
  minima of a real-valued function by asymptotic analysis of a dissipative
  dynamical system,'' \emph{Communications in Contemporary Mathematics},
  vol.~2, no.~1, pp. 1--34, 2000.

\bibitem{yi2018distributedopti}
X.~Yi, L.~Yao, T.~Yang, J.~George, and K.~H. Johansson, ``Distributed
  optimization for second-order multi-agent systems with dynamic
  event-triggered communication,'' \emph{arXiv:1803.06380v2}, 2018.

\bibitem{olfati2004consensus}
R.~Olfati-Saber and R.~M. Murray, ``Consensus problems in networks of agents
  with switching topology and time-delays,'' \emph{IEEE Transactions on
  Automatic Control}, vol.~49, no.~9, pp. 1520--1533, 2004.

\bibitem{godsil2013algebraic}
C.~Godsil and G.~F. Royle, \emph{Algebraic Graph Theory}.\hskip 1em plus 0.5em
  minus 0.4em\relax Springer Science \& Business Media, 2013, vol. 207.

\bibitem{khalil2002nonlinear}
H.~K. Khalil, \emph{Nonlinear Systems, 3rd}.\hskip 1em plus 0.5em minus
  0.4em\relax Prentice-Hall, New Jersey, 2002.

\end{thebibliography}

\appendix

\subsection{Useful Lemmas}\label{cdc18so:lemmaandproof}

For a connected graph, we have the following results.
\begin{lemma}\label{cdc18so:lemma:LKn}
(\cite[Lemma 2.1]{yi2017formation}) Let $K_n=I_n-\frac{1}{n}{\bf 1}_n{\bf 1}^{\top}_n$ and assume graph $\mathcal G$ is connected, then its Laplacian matrix $L$ is positive semi-definite, $K_n{\bf 1}_n=0$ and $K_nL=LK_n=L$. Moreover, we have
\begin{align}\label{cdc18so:lemma:LKneq}
\rho(K_n)=1~\text{and}~0\le\rho_2(L)K_n\le L.
\end{align}
\end{lemma}

\begin{lemma}\label{cdc18so:lemma:LR}
Assume graph $\mathcal G$ is connected, then there exists an orthogonal matrix $Q=\left[\begin{array}{ll}r&R
\end{array}\right]\in\mathbb{R}^{n\times n}$ with $r=\frac{1}{\sqrt{n}}{\bf 1}_n$ and $R\in\mathbb{R}^{n\times (n-1)}$ such that
\begin{align}
L=\left[\begin{array}{ll}r&R
\end{array}\right]\left[\begin{array}{ll}0&0\\
0&\Lambda_1
\end{array}\right]\left[\begin{array}{l}r^\top\\
R^\top
\end{array}\right],\label{cdc18so:lemma:LReq}\\
L^{\frac{1}{2}}=\left[\begin{array}{ll}r&R
\end{array}\right]\left[\begin{array}{ll}0&0\\
0&\sqrt{\Lambda_1}
\end{array}\right]\left[\begin{array}{l}r^\top\\
R^\top
\end{array}\right],\label{cdc18so:lemma:LRsqrteq}\\
R(\sqrt{\Lambda_1})^{-1}R^\top L=LR(\sqrt{\Lambda_1})^{-1}R^\top=K_n,\label{cdc18so:lemma:LReq2}\\
R(\Lambda_1)^{-1}R^\top L=LR(\Lambda_1)^{-1}R^\top=K_n,\label{cdc18so:lemma:LRsqrteq2}\\
\frac{1}{\rho(L)}K_n=\frac{1}{\rho(L)} RR^\top\le R(\Lambda_1)^{-1}R^\top,\label{cdc18so:lemma:LReq21}\\
R(\Lambda_1)^{-1}R^\top \le\frac{1}{\rho_2(L)} RR^\top= \frac{1}{\rho_2(L)}K_n,\label{cdc18so:lemma:LReq22}
\end{align}
where $\Lambda_1=\diag([\lambda_2,\dots,\lambda_n])$ with $0<\lambda_2\le\cdots\le\lambda_n$ are the eigenvalues of $L$, and $\sqrt{\Lambda_1}=\diag([\sqrt{\lambda_2},\dots,\sqrt{\lambda_n}])$.
\end{lemma}
\begin{proof}
\eqref{cdc18so:lemma:LReq} follows from Theorem 1 in \cite{olfati2004consensus} and Corollary 8.4.6 in \cite{godsil2013algebraic}. \eqref{cdc18so:lemma:LRsqrteq}--\eqref{cdc18so:lemma:LReq22} directly follows from \eqref{cdc18so:lemma:LReq}, $Q$ is an orthogonal matrix, and the definition of $\Lambda_1$.
\end{proof}

From Proposition 3.6 in \cite{shi2015extra}, we have the following lemma, which plays an important role in the proof of exponential convergence later.
\begin{lemma}\label{cdc18so:lemma:rsc}
Suppose that Assumptions \ref{cdc18so:ass:filc}, \ref{cdc18so:ass:fil}, and \ref{cdc18so:ass:fiug} hold, and that the underlying undirected graph $\mathcal G$ is connected, then for any $r>0$,
\begin{align*}
&(\nabla f(\bsx)-\nabla f(\bsx^*))^\top(\bsx-\bsx^*)+ r\bsx^\top (L\otimes I_p)\bsx\\
\ge& m\|\bsx-\bsx^*\|^2,~\forall \bsx\in\mathbb{R}^{np},
\end{align*}
where $\bsx=[x_1^\top,\cdots,x_n^\top]^\top$, $f(\bsx)=\sum_{i=1}^nf_i(x_i)$, $\bsx^*={\bf 1}_n \otimes x^*$, $m=\min\Big\{m_f-2M\iota,\frac{\rho_2(L)}{2r(1+\frac{1}{\iota^2})}\Big\}$, $\overline{M}=\max_{i\in\mathcal V}\{\bar{M}_i\}$, and $\iota\in(0,\frac{m_f}{2\overline{M}})$.
\end{lemma}

\subsection{Proof of Theorem~\ref{cdc18so:theorem}}\label{cdc18so:theoremproof}
The proof is carried out in three steps.

(i) In this step, we show that the global objective function achieves its minimum value at the equilibrium points of the multi-agent system \eqref{cdc18so:eqn:cta}.

By pre-multiplying \eqref{cdc18so:eqn:cta1}  with $({\bf 1}^T_n \otimes I_p)$, we obtain
\begin{align*}
\sum_{i=1}^n\dot{v}_i(t)\equiv0.
\end{align*}
Thus,
\begin{align}\label{cdc18so:vsum}
\sum_{i=1}^nv_i(t)\equiv\sum_{i=1}^nv_i(0)=0,~\forall t\ge0.
\end{align}

Consider the equilibrium points $(\bar{\bsv},\bar{\bsx},\bar{\bsy})$ of \eqref{cdc18so:eqn:cta}. They must satisfy
\begin{subequations}\label{cdc18so:ConsCond}
\begin{align}
\bar{\bsy} &=0, \label{cdc18so:ConsCond2}\\
  -\theta \bar{\bsv} - \alpha \nabla f(\bar{\bsx}) &=0, \label{cdc18so:ConsCond3}\\
  (L \otimes I_p)\bar{\bsx} &=0. \label{cdc18so:ConsCond1}
\end{align}
\end{subequations}

Since the graph is connected, it follows from \eqref{cdc18so:ConsCond1} that $\bar{\bsx}={\bf 1}_n \otimes x^0$, $x^0 \in \mathbb{R}^p$, which means that consensus is achieved.

From \eqref{cdc18so:vsum}, we have
\begin{align}
\sum_{i=1}^n\bar{v}_i&=0.\label{cdc18so:vbarsum}
\end{align}

By pre-multiplying \eqref{cdc18so:ConsCond3} with $({\bf 1}^T_n \otimes I_p)$, we obtain
\begin{align*}
\sum_{i=1}^n\nabla  f_i(x^0)=0,
\end{align*}
which means that the global optimality is achieved. Thus, the equilibrium point $x^0$ is an optimal global minimizer.

Moreover, from \eqref{cdc18so:vsum} and \eqref{cdc18so:vbarsum}, we have
\begin{align}\label{cdc18so:vvbar}
&\bar{\bsx}^\top(\bsv(t)-\bar{\bsv})=0,\nonumber\\
&(K_n\otimes I_p)(\bsv(t)-\bar{\bsv})=\bsv(t)-\bar{\bsv},~\forall t\ge0.
\end{align}

(ii) In this step, we use the Lyapunov analysis method to show that any equilibrium points of \eqref{cdc18so:eqn:cta} are globally asymptotically stable, which establishes that
the proposed distributed algorithm converges to the optimal consensus state asymptotically.

Consider
\begin{align}
W_1(\bsx)=f(\bsx)-\nabla  f(\bar{\bsx})^\top \bsx-f(\bar{\bsx}).
\end{align}
Due to Assumption~\ref{cdc18so:ass:filc},
it is easy to check that $W_1(\bsx)$ is  convex and $\min_{\bsx}W_1(\bsx)=W_1(\bar{\bsx})=0$.
The derivative of $W_1$ along the trajectories of \eqref{cdc18so:eqn:cta2} is
\begin{align}\label{cdc18so:w1dot}
\dot{W}_1=\nabla  f(\bsx)^\top \bsy-\nabla  f(\bar{\bsx})^\top \bsy.
\end{align}

Consider
\begin{align}\label{cdc18so:w2}
&W_2(\bsv,\bsx,\bsy)\nonumber\\
=&\frac{1}{2}\|\bsy\|^2+\frac{\gamma^2\varepsilon_0}{2}\|\bsx-\bar{\bsx}\|^2
+\gamma\varepsilon_0(\bsx-\bar{\bsx})^\top \bsy\nonumber\\
&+\frac{\theta\gamma\varepsilon_0}{2\beta}(\bsv-\bar{\bsv})^\top (R(\Lambda_1)^{-1}R^\top \otimes I_p)(\bsv-\bar{\bsv})\nonumber\\
&+\theta(\bsv-\bar{\bsv})^\top (K_n\otimes I_p)\bsx+\frac{\alpha\beta}{2}\bsx^\top(L\otimes I_p)\bsx,
\end{align}
where $\varepsilon_0\in(\frac{\theta}{\alpha\gamma},1)$.
Then,
\begin{align}\label{cdc18so:w2low}
&W_2(\bsv,\bsx,\bsy)\nonumber\\
=&\frac{1-\sqrt{\varepsilon_0}}{2}\|\bsy\|^2
+\frac{\gamma^2\varepsilon_0(1-\sqrt{\varepsilon_0})}{2}\|\bsx-\bar{\bsx}\|^2\nonumber\\
&+\frac{\sqrt{\varepsilon_0}}{2}\|\bsy\|^2+\frac{\gamma^2\varepsilon_0\sqrt{\varepsilon_0}}{2}
\|\bsx-\bar{\bsx}\|^2+\gamma\varepsilon_0(\bsx-\bar{\bsx})^\top y\nonumber\\
&+(\frac{\theta\gamma\varepsilon_0}{2\beta}-\frac{\theta^2}{2\alpha\beta})(\bsv-\bar{\bsv})^\top (R(\Lambda_1)^{-1}R^\top \otimes I_p)(\bsv-\bar{\bsv})\nonumber\\
&+\frac{\theta^2}{2\alpha\beta}(\bsv-\bar{\bsv})^\top (R(\Lambda_1)^{-1}R^\top \otimes I_p)(\bsv-\bar{\bsv})\nonumber\\
&+\theta(\bsv-\bar{\bsv})^\top (K_n\otimes I_p)\bsx+\frac{\alpha\beta}{2}\bsx^\top(L\otimes I_p)\bsx\nonumber\\
=&\frac{1-\sqrt{\varepsilon_0}}{2}\|\bsy\|^2
+\frac{\gamma^2\varepsilon_0(1-\sqrt{\varepsilon_0})}{2}\|\bsx-\bar{\bsx}\|^2\nonumber\\
&+\Big\|\frac{\sqrt{2}}{2}(\varepsilon_0)^{\frac{1}{4}}\bsy
+\frac{\sqrt{2}}{2}(\varepsilon_0)^{\frac{3}{4}}\gamma(\bsx-\bar{\bsx})\Big\|^2\nonumber\\
&+(\frac{\theta\gamma\varepsilon_0}{2\beta}-\frac{\theta^2}{2\alpha\beta})(\bsv-\bar{\bsv})^\top (R(\Lambda_1)^{-1}R^\top \otimes I_p)(\bsv-\bar{\bsv})\nonumber\\
&+\Big\|\frac{\theta}{\sqrt{2\alpha\beta}}(R(\sqrt{\Lambda_1})^{-1}R^\top \otimes I_p)(\bsv-\bar{\bsv})\nonumber\\
&+\frac{\sqrt{\alpha\beta}}{\sqrt{2}}(L^{\frac{1}{2}}\otimes I_p)\bsx\Big\|^2\nonumber\\
\ge&\frac{1-\sqrt{\varepsilon_0}}{2}\|\bsy\|^2
+\frac{\gamma^2\varepsilon_0(1-\sqrt{\varepsilon_0})}{2}\|\bsx-\bar{\bsx}\|^2\nonumber\\
&+(\frac{\theta\gamma\varepsilon_0}{2\beta}-\frac{\theta^2}{2\alpha\beta})(\bsv-\bar{\bsv})^\top (R(\Lambda_1)^{-1}R^\top \otimes I_p)(\bsv-\bar{\bsv})\nonumber\\
\ge&\frac{1-\sqrt{\varepsilon_0}}{2}\|\bsy\|^2
+\frac{\gamma^2\varepsilon_0(1-\sqrt{\varepsilon_0})}{2}\|\bsx-\bar{\bsx}\|^2\nonumber\\
&+\frac{1}{\rho(L)}(\frac{\theta\gamma\varepsilon_0}{2\beta}-\frac{\theta^2}{2\alpha\beta})(\bsv-\bar{\bsv})^\top (K_n \otimes I_p)(\bsv-\bar{\bsv})\nonumber\\
=&\frac{1-\sqrt{\varepsilon_0}}{2}\|\bsy\|^2
+\frac{\gamma^2\varepsilon_0(1-\sqrt{\varepsilon_0})}{2}\|\bsx-\bar{\bsx}\|^2\nonumber\\
&+\frac{1}{\rho(L)}(\frac{\theta\gamma\varepsilon_0}{2\beta}-\frac{\theta^2}{2\alpha\beta})
\|\bsv-\bar{\bsv}\|^2\ge0,
\end{align}
where the second equality holds since  \eqref{cdc18so:lemma:LReq2} and $K_nK_n=K_n$, the second inequality holds due to \eqref{cdc18so:lemma:LReq21}, the last equality holds since \eqref{cdc18so:vvbar}, and the last inequality holds due to $\sqrt{\varepsilon_0}<1$ and $\theta\gamma\varepsilon_0>\frac{\theta^2}{\alpha}$.
The derivative of $W_2$ along the trajectories of \eqref{cdc18so:eqn:cta} satisfies
\begin{align}\label{cdc18so:w2dot}
\dot{W}_2=&\bsy^\top\Big\{- \gamma \bsy-\alpha\beta (L \otimes I_p)\bsx-\theta \bsv- \alpha \nabla f(\bsx)\Big\}\nonumber\\
&+\gamma^2\varepsilon_0 (\bsx-\bar{\bsx})^\top\bsy+\gamma\varepsilon_0(\bsx-\bar{\bsx})^\top \Big\{- \gamma \bsy\nonumber\\
&-\alpha\beta (L \otimes I_p)\bsx-\theta \bsv- \alpha \nabla f(\bsx)\Big\}\nonumber\\
&+\gamma\varepsilon_0\bsy^\top \bsy+\theta\gamma\varepsilon_0(\bsv-\bar{\bsv})^\top (K_n \otimes I_p)\bsx\nonumber\\
&+\theta(\bsv-\bar{\bsv})^\top (K_n\otimes I_p)\bsy+\theta\beta \bsx^\top(L\otimes I_p)\bsx\nonumber\\
&+\alpha\beta \bsx^\top(L\otimes I_p)\bsy\nonumber\\
=&\bsy^\top\Big\{- \gamma \bsy- \theta (\bsv-\bar{\bsv})
+\alpha \nabla f(\bar{\bsx})- \alpha \nabla f(\bsx)\Big\}\nonumber\\
&+\gamma\varepsilon_0(\bsx-\bar{\bsx})^\top \Big\{
- \theta (\bsv-\bar{\bsv})+\alpha \nabla f(\bar{\bsx})- \alpha \nabla f(\bsx)\Big\}\nonumber\\
&+\gamma\varepsilon_0\bsy^\top \bsy+\theta\gamma\varepsilon_0(\bsv-\bar{\bsv})^\top \bsx\nonumber\\
&+\theta(\bsv-\bar{\bsv})^\top \bsy-(\alpha\gamma\varepsilon_0-\theta)\beta \bsx^\top(L\otimes I_p)\bsx\nonumber\\
=&-\gamma(1-\varepsilon_0)\bsy^\top \bsy+\alpha \bsy^\top(\nabla f(\bar{\bsx})-\nabla f(\bsx))\nonumber\\
&-(\alpha\gamma\varepsilon_0-\theta)\beta \bsx^\top(L\otimes I_p)\bsx\nonumber\\
&-\alpha\gamma\varepsilon_0 (\bsx-\bar{\bsx})^\top(\nabla f(\bsx)-\nabla f(\bar{\bsx})),
\end{align}
where the first equality holds since $R(\Lambda_1)^{-1}R^\top L=LR(\Lambda_1)^{-1}R^\top=K_n$ as shown in \eqref{cdc18so:lemma:LRsqrteq2}, and the second equality holds since \eqref{cdc18so:ConsCond3} and \eqref{cdc18so:vvbar}, and the last equality holds since $\bar{\bsx}^\top(L \otimes I_p)=0$.

Consider the following Lyapunov function
\begin{align}\label{cdc18so:v1}
V_1(\bsv,\bsx,\bsy)=\alpha W_1(\bsx)+W_2(\bsv,\bsx,\bsy).
\end{align}
From \eqref{cdc18so:w1dot} and \eqref{cdc18so:w2dot}, we know that the derivative of $V_1$ along the trajectories of \eqref{cdc18so:eqn:cta} satisfies
\begin{align}\label{cdc18so:v1dot}
\dot{V}_1=& -\gamma(1-\varepsilon_0)\bsy^\top \bsy-(\alpha\gamma\varepsilon_0-\theta)\beta \bsx^\top(L\otimes I_p)\bsx\nonumber\\
&-\alpha\gamma\varepsilon_0 (\bsx-\bar{\bsx})^\top(\nabla f(\bsx)-\nabla f(\bar{\bsx}))\le0.
\end{align}
Noting that $V_1(\bsv,\bsx,\bsy)$ is radially unbounded due to \eqref{cdc18so:w2low}, then by LaSalle's Invariance Principle \cite{khalil2002nonlinear},  we know that $x_i(t)$ asymptotically converges to
$\{x^1 \in \mathbb{R}^p:~\sum_{i=1}^n(x^1-\bar{x})(\nabla f_i(x^1)-\nabla f_i(x^0))=0\}$. Noting that $\sum_{i=1}^n\nabla f_i(x^0)=0$, we have $\sum_{i=1}^n(x^1-x^0)(\nabla f_i(x^1)-\nabla f_i(x^0))=0$ is equivalent to $\sum_{i=1}^n\nabla f_i(x^1)=0$, i.e., $x^1$ is a global minimizer. Thus, $x_i(t)$ asymptotically converges to one global minimizer.

(iii) In this step, we use the Lyapunov analysis method to show that the converge speed is exponential if Assumptions \ref{cdc18so:ass:fil} and \ref{cdc18so:ass:fiu} hold.

If Assumption \ref{cdc18so:ass:fil} holds, then the equilibrium point $x^0$ is the unique optimal consensus state $x^*$.
Then from \eqref{cdc18so:w2low} and \eqref{cdc18so:v1dot}, we know that \begin{align*}
&\frac{\gamma^2\varepsilon_0(1-\sqrt{\varepsilon_0})}{2}\|x_i(t)-x^*\|^2\\
&\le\frac{\gamma^2\varepsilon_0(1-\sqrt{\varepsilon_0})}{2}\|\bsx(t)-\bar{\bsx}\|^2\\
&\le V_1(\bsv(t),\bsx(t),\bsy(t))\\
&\le V_1(\bsv(0),\bsx(0),\bsy(0)),~\forall i\in\mathcal V,~\forall t\ge0.
\end{align*}
Denote
\begin{equation}\label{cdc18so:convexset}
D=\conv \Big\{a\in\mathbb{R}^p:\|a-x^*\|^2\le
\frac{2V_1(\bsv(0),\bsx(0),\bsy(0))}{\gamma^2\varepsilon_0(1-\sqrt{\varepsilon_0})}\Big\},
\end{equation}
where $\conv$ denotes the convex hull. It is straightforward to see that $D\subseteq\mathbb{R}^p$ is compact and convex, and $x_i(t)\in D,~\forall i\in\mathcal V,~\forall t\ge0$.

If Assumption \ref{cdc18so:ass:fiu} holds, then
\begin{align}
\|\nabla f(\bar{\bsx})-\nabla f(\bsx)\|\le M(D)\|\bar{\bsx}-\bsx\|,\label{cdc18so:fxbarfx}\\
W_1(\bsx)\le\frac{M(D)}{2}\|\bar{\bsx}-\bsx\|^2,~\forall \bsx\in D.\label{cdc18so:w1up}
\end{align}
Moreover, by letting $r=\frac{\alpha\gamma\varepsilon_0-\theta}{\alpha\gamma\varepsilon_0}$ and $\iota=\frac{m_f}{4M(D)}$ in Lemma~\ref{cdc18so:lemma:rsc}, we have
\begin{align}
&\alpha\gamma\varepsilon_0(\bsx-\bar{\bsx})^\top(\nabla f(\bsx)-\nabla f(\bar{\bsx}))\nonumber\\
&+(\alpha\gamma\varepsilon_0-\theta)\bsx^\top(L\otimes I_p)\bsx
\ge \alpha\gamma\varepsilon_0m_1\|\bar{\bsx}-\bsx\|^2,\label{cdc18so:fxfx}
\end{align}
where $m_1=\min\Big\{\frac{m_f}{2},\frac{\rho_2(L)m_f^2\alpha\gamma\varepsilon_0}
{2(\alpha\gamma\varepsilon_0-\theta)(m^2_f+16M^2(D))}\Big\}$.
Then, from \eqref{cdc18so:v1dot} and \eqref{cdc18so:fxfx}, we have
\begin{align}\label{cdc18so:w2dot2}
\dot{V}_1
\le& -\gamma(1-\varepsilon_0)\|\bsy\|^2-\alpha\gamma\varepsilon_0m_1\|\bsx-\bar{\bsx}\|^2\nonumber\\
\le&-\varepsilon_1\Big\{\|\bsy\|^2+\|\bsx-\bar{\bsx}\|^2\Big\},
\end{align}
where $\varepsilon_1$ is defined in \eqref{cdc18so:varepsilon1}.

Consider
\begin{align}\label{cdc18so:w3}
&W_3(\bsv,\bsx,\bsy)\nonumber\\
=&\frac{\varepsilon}{2\alpha}\|\bsy\|^2+\varepsilon(\bsv-\bar{\bsv})^\top (K_n \otimes I_p)\bsy\nonumber\\
&+\frac{\varepsilon\alpha}{2}(\bsv-\bar{\bsv})^\top (K_n \otimes I_p)(\bsv-\bar{\bsv})+\varepsilon W_1(\bsx),
\end{align}
where $\varepsilon>0$ is a constant.
The derivative of $W_3$ along the trajectories of \eqref{cdc18so:eqn:cta} satisfies
\begin{align}\label{cdc18so:w3dot}
&\dot{W}_3\nonumber\\
=&\frac{\varepsilon}{\alpha}\bsy^\top\Big\{-\gamma \bsy-\alpha\beta(L\otimes I_p)\bsx- \theta \bsv- \alpha \nabla f(\bsx)\Big\}\nonumber\\
&+\varepsilon(\bsv-\bar{\bsv})^\top (K_n \otimes I_p) \Big\{- \gamma \bsy-\alpha\beta (L \otimes I_p)\bsx\nonumber\\
&- \theta \bsv- \alpha \nabla f(\bsx)\Big\}+\varepsilon \beta \bsx^\top(L \otimes I_p)\bsy\nonumber\\
&+\varepsilon\alpha\beta(\bsv-\bar{\bsv})^\top (L \otimes I_p)\bsx+\varepsilon \bsy^\top(\nabla f(\bsx)-\nabla  f(\bar{\bsx}))\nonumber\\
=&\frac{\varepsilon}{\alpha}\bsy^\top\Big\{-\gamma \bsy- \theta (\bsv-\bar{\bsv})+\alpha \nabla f(\bar{\bsx})- \alpha \nabla f(\bsx)\Big\}\nonumber\\
&+\varepsilon(\bsv-\bar{\bsv})^\top \Big\{- \gamma \bsy-\theta (\bsv-\bar{\bsv})+\alpha \nabla f(\bar{\bsx})- \alpha \nabla f(\bsx)\Big\}\nonumber\\
&+\varepsilon \bsy^\top(\nabla f(\bsx)-\nabla  f(\bar{\bsx}))\nonumber\\
=&-\varepsilon(\gamma+\frac{\theta}{\alpha})(\bsv-\bar{\bsv})^\top \bsy-\varepsilon\theta (\bsv-\bar{\bsv})^\top(\bsv-\bar{\bsv})\nonumber\\
&-\frac{\varepsilon\gamma}{\alpha}\|\bsy\|^2+\varepsilon\alpha (\bsv-\bar{\bsv})^\top (\nabla f(\bar{\bsx})-\nabla f(\bsx))\nonumber\\
\le&\frac{\varepsilon\theta }{4}\|\bsv-\bar{\bsv}\|^2+\varepsilon\Big(\frac{1}{\theta}(\gamma+\frac{\theta}{\alpha})^2
-\frac{\gamma}{\alpha}\Big)\|\bsy\|^2\nonumber\\
&-\varepsilon\theta\|\bsv-\bar{\bsv}\|^2
+\frac{\varepsilon\theta }{4}\|\bsv-\bar{\bsv}\|^2+\frac{\varepsilon\alpha^2}{\theta }\|\nabla f(\bar{\bsx})-\nabla f(\bsx)\|^2\nonumber\\
\le&-\frac{\varepsilon\theta }{2}(\bsv-\bar{\bsv})^\top(\bsv-\bar{\bsv})+\varepsilon\varepsilon_2\Big\{\|\bsy\|^2
+\|\bsx-\bar{\bsx}\|^2\Big\},
\end{align}
where $\varepsilon_2$ is defined in \eqref{cdc18so:varepsilon2}, and  the first equality holds since $K_nL=LK_n=L$  as shown in Lemma \ref{cdc18so:lemma:LKn}, the second equality holds since \eqref{cdc18so:ConsCond3}  and \eqref{cdc18so:vvbar}, the first inequality holds since the Young's inequality, and the last inequality holds due to \eqref{cdc18so:fxbarfx}.

Consider
\begin{align*}
V_2(\bsv,\bsx,\bsy)=(1+\frac{\varepsilon\varepsilon_2}{\varepsilon_1})V_1(\bsv,\bsx,\bsy)
+W_3(\bsv,\bsx,\bsy).
\end{align*}
Noting that $\varepsilon_0<1$, similar the way to get \eqref{cdc18so:w2low}, we have
\begin{align}\label{cdc18so:w2lows}
W_2(\bsv,\bsx,\bsy)\ge \frac{\gamma^2\varepsilon_0(1-\varepsilon_0)}{2}\|\bsx-\bar{\bsx}\|^2.
\end{align}
Then, from  \eqref{cdc18so:w2lows}, we have
\begin{align}\label{cdc18so:v2low}
V_2(\bsv,\bsx,\bsy)
\ge(1+\frac{\varepsilon\varepsilon_2}{\varepsilon_1})W_2(\bsv,\bsx,\bsy)
\ge\tilde{\varepsilon}_1\|\bsx-\bar{\bsx}\|^2\ge0,
\end{align}
where $\tilde{\varepsilon}_1=(1+\frac{\varepsilon\varepsilon_2}{\varepsilon_1})
\frac{\gamma^2\varepsilon_0(1-\varepsilon_0)}{2}$.
From \eqref{cdc18so:w2dot2} and \eqref{cdc18so:w3dot}, we know that the derivative of $V_2$ along the trajectories of \eqref{cdc18so:eqn:cta} satisfies
\begin{align}\label{cdc18so:v2dot}
\dot{V}_2\le& -\varepsilon_1\Big\{\|\bsy\|^2+\|\bsx-\bar{\bsx}\|^2\Big\}
-\frac{\varepsilon\theta}{2}\|\bsv-\bar{\bsv}\|^2\nonumber\\
\le& -\varepsilon_3\Big\{\|\bsy\|^2+\|\bsx-\bar{\bsx}\|^2+\|\bsv-\bar{\bsv}\|^2\Big\},
\end{align}
where $\varepsilon_3$ is defined in \eqref{cdc18so:varepsilon3}.

From  the Young's inequality, \eqref{cdc18so:lemma:LReq22}, and  \eqref{cdc18so:lemma:LKneq},  we have
\begin{align}
W_2\le& \|\bsy\|^2+\gamma^2\varepsilon_0\|\bsx-\bar{\bsx}\|^2
+\frac{\theta\gamma\varepsilon_0}{\beta\rho_2(L)}\|\bsv-\bar{\bsv}\|^2\nonumber\\
&+\alpha\beta\rho(L)\|\bsx-\bar{\bsx}\|^2,\label{cdc18so:w2up}\\
W_3\le& \frac{\varepsilon}{\alpha}\|\bsy\|^2+\varepsilon\alpha\|\bsv-\bar{\bsv}\|^2+\varepsilon W_1(\bsx).\label{cdc18so:w3up}
\end{align}
Then, from \eqref{cdc18so:w2up}, \eqref{cdc18so:w3up},  and \eqref{cdc18so:w1up}, we have
\begin{align}\label{cdc18so:v2up}
V_2=&(1+\frac{\varepsilon\varepsilon_2}{\varepsilon_1})(\alpha W_1+W_2)+W_3\nonumber\\
\le&\varepsilon_4\Big\{\|\bsy\|^2++\|\bsx-\bar{\bsx}\|^2+\|\bsv-\bar{\bsv}\|^2\Big\},
\end{align}
where $\varepsilon_4$ is defined in \eqref{cdc18so:varepsilon4}.
Then,
\begin{align}\label{cdc18so:v2dot2}
\dot{V}_2(t)\le& -\frac{\varepsilon_3}{\varepsilon_4}V_2(t),~\forall t\ge0.
\end{align}
From \eqref{cdc18so:v2low} and \eqref{cdc18so:v2dot2},
we know that
\begin{align*}
&\|x_i(t)-x^*\|\le\|\bsx(t)-\bar{\bsx}\| \\
\le& \sqrt{\frac{1}{\tilde{\varepsilon}_1}V_2(\bsv(t),\bsx(t),\bsy(t))}\\
\le&\sqrt{\frac{1}{\tilde{\varepsilon}_1}V_2(\bsv(0),\bsx(0),\bsy(0))}
e^{-\frac{\varepsilon_3}{2\varepsilon_4}t},~\forall i\in\mathcal V,~\forall t\ge0,
\end{align*}
i.e., $x_i(t)$ exponentially converges to the unique global minimizer $x^*$ with a rate no less than $\frac{\varepsilon_3}{2\varepsilon_4}>0$.


\subsection{Proof of Theorem~\ref{cdc18so:theoremvevent}}\label{cdc18so:theoremveventproof}
The proof is carried out in two parts.

(i) In this part, we show that there is no Zeno behavior by contradiction.

We first note that it follows from the way we determine the triggering times by \eqref{cdc18so:dynamicv} that
\begin{align}\label{cdc18so:dynamicv2}
\kappa_i\Big(\|e^{x}_i(t)\|^2 -\frac{(\alpha\gamma\varepsilon_0-\theta)\beta\sigma_i}{4\varphi_i}\hat{q}_i(t)\Big)\le \chi_i(t),~\forall t\ge0.
\end{align}
This together with \eqref{cdc18so:chi} implies that
$$\dot{\chi}_i(t)\ge -\phi_i\chi_i(t)-\frac{\delta_i}{\kappa_i}\chi_i(t),~\forall t\ge0.$$
Therefore,
\begin{align*}
\chi_i(t)\ge\chi_i(0)e^{-(\phi_i+\frac{\delta_i}{\kappa_i})t}>0,~\forall t\ge0.
\end{align*}

Assume that there exists Zeno behavior. Then there exists an agent $i$, such that $\lim_{k\rightarrow+\infty}t^i_k=T_0$ where $T_0$ is a positive constant. Noting that  $y_i(t)$ is continuous with respect to time $t$, there exists a positive constant $M_0$ such that $\|\dot{x}_i(t)\|=\|y_i(t)\|\le M_0,~\forall t\in[0,T_0]$. Thus, $\frac{d\|e^{x}_i(t)\|}{dt}\le\|\dot{e}^{x}_i(t)\|=\|\dot{x}_i(t)\|\le M_0,~\forall t\in[0,T_0]$. The rest of the proof to show that there exists a contradiction under the assumption that there exists Zeno behavior is similar to the proof in \cite{yi2017distributed}. Thus, we omit the details here. Therefore, Zeno behavior is excluded.

(ii) In this part, we use the Lyapunov analysis method to show that each individual solution $x_i(t)$ exponentially converges to the minimizer $x^*$

Consider $V_1(\bsv,\bsx,\bsy)$ defined in \eqref{cdc18so:v1} again.
The derivative of $V_1$ along the trajectories of \eqref{cdc18so:eqn:ctaevent} satisfies
\begin{align}\label{cdc18so:w2dotevent}
&\dot{V}_1\nonumber\\
=&\bsy^\top\Big\{- \gamma \bsy-\alpha\beta (L \otimes I_p)\hat{\bsx}-\theta  \bsv- \alpha \nabla f(\bsx)\Big\}\nonumber\\
&+\gamma^2\varepsilon_0 (\bsx-\bar{\bsx})^\top \bsy+\gamma\varepsilon_0(\bsx-\bar{\bsx})^\top \Big\{- \gamma \bsy\nonumber\\
&-\alpha\beta (L \otimes I_p)\hat{\bsx}- \theta \bsv- \alpha \nabla f(\bsx)\Big\}\nonumber\\
&+\gamma\varepsilon_0\bsy^\top \bsy+\theta\gamma\varepsilon_0 (\bsv-\bar{\bsv})^\top (K_n \otimes I_p)\hat{\bsx}\nonumber\\
&+\theta(\bsv-\bar{\bsv})(K_n\otimes I_p)\bsy+\theta\beta \bsx^\top(L\otimes I_p)\hat{\bsx}\nonumber\\
&+\alpha\beta \bsx^\top(L\otimes I_p)\bsy-\alpha \bsy^\top(\nabla f(\bar{\bsx})-\nabla f(\bsx))\nonumber\\
=&\bsy^\top\Big\{- \gamma \bsy-\theta (\bsv-\bar{\bsv})+\alpha \nabla f(\bar{\bsx})- \alpha \nabla f(\bsx)\Big\}\nonumber\\
&+\gamma\varepsilon_0(\bsx-\bar{\bsx})^\top \Big\{-\alpha\beta (L \otimes I_p)\hat{\bsx}\nonumber\\
&-\theta (\bsv-\bar{\bsv})+\alpha \nabla f(\bar{\bsx})- \alpha \nabla f(\bsx)\Big\}\nonumber\\
&+\gamma\varepsilon_0\bsy^\top \bsy+\theta\gamma\varepsilon_0 (\bsv-\bar{\bsv})^\top(\hat{\bsx}-\bar{\bsx})\nonumber\\
&+\theta(\bsv-\bar{\bsv})\bsy+\theta\beta \bsx^\top(L\otimes I_p)\hat{\bsx}\nonumber\\
&-\alpha\beta (\bse^\bsx)^\top(L\otimes I_p)\bsy-\alpha \bsy^\top(\nabla f(\bar{\bsx})-\nabla f(\bsx))\nonumber\\
=&-\gamma(1-\varepsilon_0)\bsy^\top \bsy-(\alpha\gamma\varepsilon_0-\theta)\beta\hat{\bsx}^\top (L \otimes I_p)\hat{\bsx}\nonumber\\
&+(\alpha\gamma\varepsilon_0-\theta)\beta(\bse^{\bsx})^\top (L \otimes I_p)\hat{\bsx}+\gamma \theta\varepsilon_0 (\bse^{\bsx})^\top(\bsv-\bar{\bsv})\nonumber\\
&-\alpha (\bsx-\bar{\bsx})^\top(\nabla f(\bsx)-\nabla f(\bar{\bsx}))-\alpha\beta (\bse^\bsx)^\top(L\otimes I_p)\bsy\nonumber\\
\le&-\gamma(1-\varepsilon_0)\bsy^\top \bsy-(\alpha\gamma\varepsilon_0-\theta)\beta\hat{\bsx}^\top (L \otimes I_p)\hat{\bsx}\nonumber\\
&-\frac{(\alpha\gamma\varepsilon_0-\theta)\beta}{4}\sum_{i=1}^n\sum_{j=1}^nL_{ij}
\|\hat{x}_j-\hat{x}_i\|^2\nonumber\\
&+(\alpha\gamma\varepsilon_0-\theta)\beta\sum_{i=1}^{n}L_{ii}\|e^{x}_i\|^2+\frac{\gamma^2\theta
\varepsilon_0^2 }{4\varepsilon_8}\|\bse^{\bsx}\|^2\nonumber\\
&+\varepsilon_8\theta\|\bsv-\bar{\bsv}\|^2+\frac{\gamma(1-\varepsilon_0)}{2}\bsy^\top \bsy
\nonumber\\
&-\alpha (\bsx-\bar{\bsx})^\top(\nabla f(\bsx)-\nabla f(\bar{\bsx}))\nonumber\\
&+\frac{\alpha^2\beta^2}{\gamma(1-\varepsilon_0)}\sum_{i=1}^n\Big(L_{ii}-\sum_{j=1,j\neq i}^nL_{jj}L_{ij}\Big)\|e^{x}_i\|^2\nonumber\\
=&-\frac{\gamma(1-\varepsilon_0)}{2}\bsy^\top \bsy
-\frac{(\alpha\gamma\varepsilon_0-\theta)\beta}{2}\hat{\bsx}^\top (L \otimes I_p)\hat{\bsx}\nonumber\\
&-\alpha (\bsx-\bar{\bsx})^\top(\nabla f(\bsx)-\nabla f(\bar{\bsx}))\nonumber\\
&+\varepsilon_8\theta\|\bsv-\bar{\bsv}\|^2+
\sum_{i=1}^{n}\tilde\varphi_i\|e^{x}_i\|^2\nonumber\\
\le&-\frac{\gamma(1-\varepsilon_0)}{2}\bsy^\top \bsy
-\frac{(\alpha\gamma\varepsilon_0-\theta)\beta}{4}\hat{\bsx}^\top (L \otimes I_p)\hat{\bsx}\nonumber\\
&-m_2\alpha\|\bsx-\bar{\bsx}\|^2+\varepsilon_8\theta\|\bsv-\bar{\bsv}\|^2+
\sum_{i=1}^{n}\varphi_i\|e^{x}_i\|^2\nonumber\\
\le&-\varepsilon_5(\|\bsy\|^2+\|\bsx-\bar{\bsx}\|^2)
-\frac{(\alpha\gamma\varepsilon_0-\theta)\beta}{2}\hat{\bsx}^\top (L \otimes I_p)\hat{\bsx}\nonumber\\
&+\varepsilon_8\theta \|\bsv-\bar{\bsv}\|^2+
\sum_{i=1}^{n}\varphi_i\|e^{x}_i\|^2,
\end{align}
where $\tilde{\varphi}_i=(\alpha\gamma\varepsilon_0-\theta)\beta L_{ii}+\frac{\gamma^2\theta
\varepsilon_0^2 }{4\varepsilon_8}
+\frac{\alpha^2\beta^2}{\gamma(1-\varepsilon_0)}\Big(L_{ii}-\sum_{j=1,j\neq i}^nL_{jj}L_{ij}\Big)$, and $m_2$, $\varepsilon_5$, and $\varepsilon_8$ are defined in \eqref{cdc18so:m2}, \eqref{cdc18so:varepsilon5}, and \eqref{cdc18so:varepsilon8}, respectively; the last equality holds due to the property that $\sum_{i=1}^{n}\hat{q}_i(t)=\hat{\bsx}^\top (L \otimes I_p) \hat{\bsx}$; the first inequality holds due to Young's inequality and the following inequalities
\begin{align*}
&(\bse^{\bsx})^\top (L \otimes I_p)\hat{\bsx}
=\sum_{i=1}^n\sum_{j=1}^nL_{ij}(e^{x}_i)^\top\hat{x}_j\\
=&\sum_{i=1}^n\sum_{j=1,j\neq i}^nL_{ij}(e^{x}_i)^\top(\hat{x}_j-\hat{x}_i)\\
\le&-\sum_{i=1}^n\sum_{j=1,j\neq i}^nL_{ij}\Big\{\|e^{x}_i\|^2+\frac{1}{4}\|\hat{x}_j-\hat{x}_i\|^2\Big\}\\
=&\sum_{i=1}^nL_{ii}\|e^{x}_i\|^2
-\frac{1}{4}\sum_{i=1}^n\sum_{j=1}^nL_{ij}\|\hat{x}_j
-\hat{x}_i\|^2,
\end{align*}

\begin{align*}
-\alpha\beta (\bse^\bsx)^\top(L\otimes I_p)\bsy\le&\frac{\alpha^2\beta^2}{2\gamma(1-\varepsilon_0)}(\bse^\bsx)^\top(L^2\otimes I_p)\bse^\bsx\\
&+\frac{\gamma(1-\varepsilon_0)}{2}\bsy^\top \bsy,
\end{align*}
and
\begin{align*}
&(\bse^\bsx)^\top(L^2\otimes I_p)\bse^\bsx=\sum_{i=1}^n\Big\|\sum_{j=1,j\neq i}^nL_{ij}(e^{x}_j-e^{x}_i)\Big\|^2\\
\le&\sum_{i=1}^n\Big\{\sum_{j=1,j\neq i}^n(-L_{ij})\sum_{j=1,j\neq i}^n(-L_{ij})\|e^{x}_j-e^{x}_i\|^2\Big\}\\
\le&-2\sum_{i=1}^nL_{ii}\sum_{j=1,j\neq i}^nL_{ij}(\|e^{x}_j\|^2+\|e^{x}_i\|^2)\\
=&\sum_{i=1}^n2(L_{ii}-\sum_{j=1,j\neq i}^nL_{jj}L_{ij})\|e^{x}_i\|^2;
\end{align*}
and the second inequality holds due to \eqref{cdc18so:hatx} and \eqref{cdc18so:fxfxevent} in the following.

Noting that
\begin{align*}
&\bsx^\top(L\otimes I_p)\bsx\le2\hat{\bsx}^\top(L\otimes)\hat{\bsx}+2(\bse^\bsx)^\top(L\otimes)\bse^\bsx\\
\le&2\hat{\bsx}^\top(L\otimes)\hat{\bsx}+2\sum_{i=1}^nL_{ii}\|e^x_i\|^2,
\end{align*}
we have
\begin{align}\label{cdc18so:hatx}
&-\frac{(\alpha\gamma\varepsilon_0-\theta)\beta}{4}\hat{\bsx}^\top(L\otimes)\hat{\bsx}
\le-\frac{(\alpha\gamma\varepsilon_0-\theta)\beta}{8}\bsx^\top(L\otimes I_p)\bsx\nonumber\\
&~~~~~~~~~~~~~~~~~~~~~~~~~+\frac{(\alpha\gamma\varepsilon_0-\theta)\beta}{4}\sum_{i=1}^nL_{ii}\|e^x_i\|^2.
\end{align}
By letting $r=\frac{(\alpha\gamma\varepsilon_0-\theta)\beta}{8\alpha}$ and $\iota=\frac{m_f}{4\overline{M}}$ in Lemma \ref{cdc18so:lemma:rsc}, we have
\begin{align}
&\alpha(\bsx-\bar{\bsx})^\top(\nabla f(\bsx)-\nabla f(\bar{\bsx}))\nonumber\\
&+\frac{(\alpha\gamma\varepsilon_0-\theta)\beta}{8}\bsx^\top(L\otimes I_p)\bsx
\ge \alpha\gamma m_2\|\bar{\bsx}-\bsx\|^2,\label{cdc18so:fxfxevent}
\end{align}
where $m_2$ is defined in \eqref{cdc18so:m2}.

Consider $W_3(\bsv,\bsx,\bsy)$ defined in \eqref{cdc18so:w3} again.
The derivative of $W_3$ along the trajectories of \eqref{cdc18so:eqn:ctaevent1} and \eqref{cdc18so:eqn:ctaevent3} satisfies
\begin{align}\label{cdc18so:w3dotevent}
&\dot{W}_3\nonumber\\
=&\frac{\varepsilon}{\alpha}\bsy^\top\Big\{-\gamma \bsy-\alpha\beta(L\otimes I_p)\hat{\bsx}- \theta \bsv- \alpha \nabla f(\bsx)\Big\}\nonumber\\
&+\varepsilon(\bsv-\bar{\bsv})^\top (K_n \otimes I_p) \Big\{- \gamma \bsy-\alpha\beta (L \otimes I_p)\hat{\bsx}\nonumber\\
&- \theta \bsv- \alpha \nabla f(\bsx)\Big\}+\varepsilon \beta \hat{\bsx}^\top(L \otimes I_p)\bsy\nonumber\\
&+\varepsilon\alpha\beta(\bsv-\bar{\bsv})^\top (L \otimes I_p)\hat{\bsx}+\varepsilon \bsy^\top(\nabla f(\bsx)-\nabla  f(\bar{\bsx}))\nonumber\\
=&\frac{\varepsilon}{\alpha}\bsy^\top\Big\{-\gamma \bsy- \theta (\bsv-\bar{\bsv})+\alpha \nabla f(\bar{\bsx})- \alpha \nabla f(\bsx)\Big\}\nonumber\\
&+\varepsilon(\bsv-\bar{\bsv})^\top \Big\{- \gamma \bsy-\theta (\bsv-\bar{\bsv})+\alpha \nabla f(\bar{\bsx})- \alpha \nabla f(\bsx)\Big\}\nonumber\\
&+\varepsilon \bsy^\top(\nabla f(\bsx)-\nabla  f(\bar{\bsx}))\nonumber\\
\le&-\frac{\varepsilon\theta }{2}(\bsv-\bar{\bsv})^\top(\bsv-\bar{\bsv})+\varepsilon\varepsilon_6\Big\{\|\bsy\|^2+\|\bsx-\bar{\bsx}\|^2\Big\},
\end{align}
where $\varepsilon_6$ is defined in \eqref{cdc18so:varepsilon6}.

Consider
\begin{align*}
W_4(\bsv,\bsx,\bsy)=\varepsilon_7V_1(\bsv,\bsx,\bsy)+W_3(\bsv,\bsx,\bsy).
\end{align*}
From \eqref{cdc18so:w2dotevent}, \eqref{cdc18so:w3dotevent}, and $\varepsilon_7=1+\frac{\varepsilon\varepsilon_6}{\varepsilon_5}$, we know that the derivative of $W_4$ along the trajectories of \eqref{cdc18so:eqn:ctaevent} satisfies
\begin{align}\label{cdc18so:w4dotevent}
&\dot{W}_4\nonumber\\
\le&-\varepsilon_5\Big\{\|\bsy\|^2+\|\bsx-\bar{\bsx}\|^2\Big\}
-\varepsilon_7\frac{(\alpha\gamma\varepsilon_0-\theta)\beta}{4}\hat{\bsx}^\top (L \otimes I_p)\hat{\bsx}\nonumber\\
&+\varepsilon_7\varepsilon_8\theta \|\bsv-\bar{\bsv}\|^2+\sum_{i=1}^{n}\varepsilon_7\varphi_i\|e^{x}_i\|^2
-\frac{\varepsilon\theta}{2}\|\bsv-\bar{\bsv}\|^2\nonumber\\
=&-\varepsilon_5\Big\{\|\bsy\|^2+\|x-\bar{\bsx}\|^2\Big\}
-\varepsilon_7\frac{(\alpha\gamma\varepsilon_0-\theta)\beta}{4}\hat{\bsx}^\top (L \otimes I_p)\hat{\bsx}\nonumber\\
&+\sum_{i=1}^{n}\varepsilon_7\varphi_i\|e^{x}_i\|^2
-\frac{\varepsilon\theta}{4}\|\bsv-\bar{\bsv}\|^2,
\end{align}
where  the equality holds since $\varepsilon_8=\frac{\varepsilon}{4\varepsilon_7}$.

Denote ${\boldsymbol{\chi}}=[\chi_1,\cdots,\chi_n]^\top$.
Consider the following Lyapunov function
\begin{align*}
V_3(\bsv,\bsx,\bsy,{\boldsymbol{\chi}})=W_4(\bsv,\bsx,\bsy)+\varepsilon_7\sum_{i=1}^n\varphi_i\chi_i.
\end{align*}

This together with \eqref{cdc18so:dynamicv2} and \eqref{cdc18so:w4dotevent} imply that the derivative of $V_3$ along the trajectories of \eqref{cdc18so:eqn:ctaevent} and \eqref{cdc18so:chi} satisfies
\begin{align}\label{cdc18so:v3dot}
&\dot{V}_3\nonumber\\
\le&-\varepsilon_5\Big\{\|\bsy\|^2+\|\bsx-\bar{\bsx}\|^2\Big\}
-\varepsilon_7\frac{(\alpha\gamma\varepsilon_0-\theta)\beta}{4}\hat{\bsx}^\top (L \otimes I_p)\hat{\bsx}\nonumber\\
&+\sum_{i=1}^{n}\varepsilon_7\varphi_i\|e^{x}_i\|^2
-\frac{\varepsilon\theta}{4}\|\bsv-\bar{\bsv}\|^2-\sum_{i=1}^n\varepsilon_7\varphi_i\phi_i \chi_i(t)\nonumber\\
&-\sum_{i=1}^n\varepsilon_7\varphi_i\delta_i \Big(\|e^{x}_i(t)\|^2 -\frac{(\alpha\gamma\varepsilon_0-\theta)\beta\sigma_i}{4\varphi_i}\hat{q}_i(t)\Big)\nonumber\\
\le&-\varepsilon_5\Big\{\|\bsy\|^2+\|\bsx-\bar{\bsx}\|^2\Big\}\nonumber\\
&-\frac{\varepsilon\theta }{4}\|\bsv-\bar{\bsv}\|^2
-\sum_{i=1}^n\varepsilon_7\varphi_i\phi_i \chi_i(t)\nonumber\\
&+\sum_{i=1}^n\varepsilon_7\varphi_i(1-\delta_i) \Big(\|e^{x}_i(t)\|^2 -\frac{(\alpha\gamma\varepsilon_0-\theta)\beta\sigma_i}{4\varphi_i}\hat{q}_i(t)\Big)\nonumber\\
\le&-\varepsilon_5\Big\{\|\bsy\|^2+\|\bsx-\bar{\bsx}\|^2\Big\}-\frac{\varepsilon\theta }{4}\|\bsv-\bar{\bsv}\|^2\nonumber\\
&
-\sum_{i=1}^n\varepsilon_7\varphi_ik_d\chi_i(t)\nonumber\\
\le&-\varepsilon_9\Big\{\|\bsy\|^2+\|\bsx-\bar{\bsx}\|^2+\|\bsv-\bar{\bsv}\|^2
+\sum_{i=1}^n\varepsilon_7\varphi_i\chi_i(t)\Big\},
\end{align}
where $\varepsilon_9$ is defined in \eqref{cdc18so:varepsilon9}.

Similar to \eqref{cdc18so:v2low}, we have
\begin{align*}
W_4(\bsv,\bsx,\bsy)\ge \tilde{\varepsilon}_2\|\bsx-\bar{\bsx}\|^2\ge0,
\end{align*}
where $\tilde{\varepsilon}_2=\varepsilon_7
\frac{\gamma^2\varepsilon_0(1-\varepsilon_0)}{2}$.
Thus,
\begin{align}\label{cdc18so:v3low}
V_3(\bsv,\bsx,\bsy,\chi)>W_4(\bsv,\bsx,\bsy)\ge \tilde{\varepsilon}_2\|\bsx-\bar{\bsx}\|^2\ge0.
\end{align}

Similar to \eqref{cdc18so:v2up}, we have
\begin{align*}
W_4
\le&\varepsilon_{10}(\|\bsy\|^2+\|\bsx-\bar{\bsx}\|^2+\|\bsv-\bar{\bsv}\|^2),
\end{align*}
where $\varepsilon_{10}$ is defined in \eqref{cdc18so:varepsilon10}.
Thus,
\begin{align*}
V_3
\le&\varepsilon_{10}\Big\{\|\bsy\|^2+\|\bsx-\bar{\bsx}\|^2+\|\bsv-\bar{\bsv}\|^2
+\sum_{i=1}^n\varepsilon_{7}\varphi_i\chi_i(t)\Big\}.
\end{align*}
Then,
\begin{align}\label{cdc18so:v3dot3}
\dot{V}_3(t)\le& -\frac{\varepsilon_9}{\varepsilon_{10}}V_3(t),~\forall t\ge0.
\end{align}
From \eqref{cdc18so:v3low} and \eqref{cdc18so:v3dot3},
we know that
\begin{align*}
&\|x_i(t)-x^*\|\le\|\bsx(t)-\bar{\bsx}\|\\
\le&\sqrt{\frac{1}
{\tilde{\varepsilon}_2}V_3(\bsv(t),\bsx(t),\bsy(t),{\boldsymbol{\chi}}(t))}\\
\le&\sqrt{\frac{1}
{\tilde{\varepsilon}_2}V_3(\bsv(0),\bsx(0),\bsy(0),{\boldsymbol{\chi}}(0))}
e^{-\frac{\varepsilon_9}{2\varepsilon_{10}}t},~\forall i\in\mathcal V,~\forall t\ge0,
\end{align*}
i.e., $x_i(t)$ exponentially converges to the unique global minimizer $x^*$ with a rate no less than $\frac{\varepsilon_9}{2\varepsilon_{10}}>0$.

\end{document}